\documentclass[11pt]{amsart}
\usepackage{etex}
\usepackage{diagbox}
\usepackage{array}
\usepackage{graphicx, pictex}
\usepackage{stmaryrd}
\usepackage{changepage}
\usepackage{float}
\restylefloat{table}
\usepackage{multirow}
\usepackage[dvipsnames]{xcolor}
\usepackage{amsmath,amssymb,amsthm,mathrsfs}
\allowdisplaybreaks

\renewcommand{\d}{\partial}

\def\d{\Omega}

\def\du#1#2#3{\overset{#3}{\underset{#2}{#1}}}
\def\Forall{\quad \hbox{ for all }}

\def\M{{\mathcal{M}}}
\def\B{{\mathcal{B}}}

\newcommand{\tn}[1]{\lVert\kern-1pt\lvert{#1}\rvert\kern-1pt\rVert}
\def\<{{\langle}}
\def\>{{\rangle}}
\def\Forall{\quad \hbox{ for all }}

\def\div{\operatorname{div}}

\def\div{\operatorname{div}}

\def\div{\operatorname{div}}

\def\d{\Omega}

\def\Forall{\quad \hbox{ for all }}

\def\d{\Omega}

\def\Forall{\quad \hbox{ for all }}

\def\tb#1{{\|\kern-1pt| #1 \|\kern-1pt|}}
\def\nm2#1#2{\|#1\|_{2,\d_{#2}}}

\def\R{\mathbb{R}}

 \theoremstyle{plain}
 \newtheorem{thm}{Theorem}[section]
 \numberwithin{equation}{section} 
 \numberwithin{figure}{section} 
 \theoremstyle{plain}
 \newtheorem{prop}[thm]{Proposition}
 \theoremstyle{plain}
 \theoremstyle{plain}
 \newtheorem{theorem}[thm]{Theorem}
 \theoremstyle{plain}
 \newtheorem{corollary}[thm]{Corollary}
\theoremstyle{plain}
 
 \theoremstyle{plain}
 
\def\M{{\mathcal{M}}}

\def\d{{\Omega}}

\def\Forall{\quad \hbox{ for all }}

\def\div{\operatorname{div}}

\def\<{{\langle}}
\def\>{{\rangle}}
\def\R{\mathbb{R}}

\def\du#1#2#3{\overset{#3}{\underset{#2}{#1}}}

\def\div{\operatorname{div}}

\def\div{\operatorname{div}}

\usepackage{amssymb}

\begin{document}

\title[Notes on FE discretization for Convectio-Diffusion]
{Notes on Finite Element Discretization  for a Model Convection-Diffusion Problem}

\author{Constantin Bacuta}
\address{University of Delaware,
Mathematical Sciences,
501 Ewing Hall, Newark, DE 19716}
\email{bacuta@udel.edu}

\author{Daniel Hayes}
\address{University of Delaware,
Department of Mathematics,
501 Ewing Hall 19716}
\email{dphayes@udel.edu}

\author{Tyler O'Grady}
\address{University of Delaware,
Department of Mathematics,
501 Ewing Hall 19716}
\email{togrady@@udel.edu}


\thanks{The work was supported  by NSF-DMS 2011615}%

\begin{abstract}

We present recent finite element  numerical  results on  a model convection-diffusion problem in the singular perturbed case when the convection term dominates the problem.   We compare  the standard Galerkin discretization using the linear element with a saddle point least square discretization that uses quadratic test functions, trying to control and explain the non-physical oscillations of the discrete solutions. We also relate  the  up-winding Petrov-Galerkin method and  the stream-line diffusion discretization method, by emphasizing the resulting  linear  systems and by comparing appropriate error norms.  Some  results can  be extended to the multidimensional case in order to come up with efficient approximations for more general singular perturbed problems, including convection dominated models. 

\end{abstract}
\maketitle

\section{Introduction}


We consider   the model  singularly perturbed  convection-reaction-diffusion problem: Find $u$ defined on $\Omega$ such that 
\begin{equation}\label{PDE_RD}
   \left\{
\begin{array}{rcl}
     -\varepsilon\, \Delta u + b\cdot \nabla u +  cu=\ f & \mbox{in} \ \ \ \Omega,\\
      u =\ 0 & \mbox{on} \ \partial\Omega,\\ 
\end{array} 
\right. 
\end{equation}
for $\varepsilon>0$, $\div b =0$,  and $c(x) \geq c_0 >0$ on $\Omega$, a bounded domain in $\Omega \subset\R^d$. 

A variational formulation of \eqref{PDE_RD} is: Find $u\in H_0^1(\Omega)$ such that 
\begin{equation}\label{eq:var1}
{\varepsilon\, (\nabla u, \nabla v)  + (b \cdot \nabla u, v) +(cu,v)}  =  {(f, v)}  \Forall v  \in  H^1_0(\Omega).
\end{equation}

The simplified one dimensional version of \eqref{PDE_RD} with $b=1$ and $c=0$ is: 
Find $u=u(x)$  on $[0, 1]$ such that
\begin{equation}\label{eq:1d-model}
\begin{cases}-\varepsilon\,  u''(x)+u'(x)=f(x),& 0<x<1\\ u(0)=0, \ u(1)=0. \end{cases}
\end{equation} 
We will assume that  the problem reaction is dominated, i.e.,  $\varepsilon \ll 1$ and  $f$ is square integrable on $[0, 1]$. 

In what follows, we  will use the following  notation:
\[ 
\begin{aligned}
a_0(u, v) &= \int_0^1 u'(x) v'(x) \, dx, \ \text{and} \ (f, v) = \int_0^1  f(x) v(x) \, dx,\ \text{and}\\
b(v, u) &=\varepsilon\, a_0(u, v)+(u',v)  \ \text{for all} \ u,v \in V:=H^{1}_0(0,1).
\end{aligned}
 \]
The variational formulation of \eqref{eq:1d-model} is: Find $u \in V:= H_0^1(0,1)$ such that
 \begin{equation}\label{VF1d}
b(v,u) = (f, v), \ \text{for all} \ v \in V.
\end{equation}
The PDE model \eqref{eq:1d-model}, and specially its multi-dimensional  extension  \eqref{PDE_RD}, arise in solving  practical problems such as  heat transfer problems in thin domains, as well as when using small step sizes in implicit time discretizations of parabolic reaction diffusion type problems, see e.g., \cite{Lin-Stynes12} and the references in  \cite{roos-stynes-tobiska-96}. The solutions to these problems are characterized by  boundary layers 
\cite{Roos-Schopf15}, 
which pose  numerical challenges due to the $\varepsilon$-dependence of the  error estimates  and of the  stability constants. 

The goal of this work is to illustrate  some  challenges  of the finite element discretization of the one dimensional model reaction diffusion problem and to emphasize on the mixed formulation and discretization advantages. We hope that  ideas, concepts, or methods we present, can be extended  to the the multidimensional case of convection dominated problems of type  \eqref{PDE_RD}. 

  Saddle Point Least Squares (SPLS)  discretizatrion as presented \cite{BJ-nc, BJprec, BJ-AA,  BQ15, Dahmen-Welper-Cohen12, BVZ11}  were used before for singularly perturbed problems in order to  improve the stability and the  rate of convergence of the discrete solutions in  special norms. 
The   SPLS approach  uses an auxiliary variable that represents the residual of the original variational formulation on the test space and another simple equation involving the residual variable  that leads to a (square) symmetric saddle point system that is more suitable for analysis and discretization. The idea is similar to the Lagrange multiplier approach, with the exception that the Lagrange multiplier here is the variable of interest. 
The SPLS method  or its variants, such as the Discontinuous Petrov–Galerkin (DPG) method,  was used efficiently for  other mixed variational  problems, see e.g., \cite{BM12,Dahmen-Welper-Cohen12,demkowicz, Barrett-Morton81}.  Many of the aspects  regarding  SPLS formulation  are common to both the DPG approach \cite{ bouma-gopalakrisnan-harb2014,car-dem-gop16,demkowicz-gopalakrishnanDPG10, demkowicz2011class,J5onDPG, dem-fuh-heu-tia19}  and the SPLS approach developed in  \cite{BJ-nc, BJprec, BJ-AA, BQ15}.

The paper is organized as follows.  We review the main ideas  of the SPLS approach in an abstract general setting in Section \ref{sec:ReviewLSPP}. In Section \ref{sec:abstract-discretization}, we present the SPLS discretization together with some general  error approximation results. We include here a new approximation result for the Petrov-Galerkin case when the norm on the continuous and discrete test spaces could be different. Section \ref{sec:1d-lin-discrete} deals with a review of four know discretization methods that have  $C^0-P^1$ as trial space  and can be viewed as mixed methods. We illustrate with plots of the discrete solutions the non-physical oscillation phenomena  for the standard and SPLS discretization and  emphasize  the strong connection between a Petrov-Galerkin (PG) and  the stream-line diffusion (SD) methods. Numerical results are  presented in Section \ref{sec:NR}.

\section{The notation and the general SPLS approach}\label{sec:ReviewLSPP}
We now review  the main ideas and concepts for the  SPLS method for a  general  mixed variational formulation.
We follow the  Saddle Point Least Squares (SPLS)  terminology that was introduced in in  \cite{BJprec, BJ-AA,BJ-nc, BQ15}. 

\subsection{The abstract variational formulation at the continuous level}  \label{subsec:mh}
We consider the (mixed) Petrov-Galerkin formulation of the more general abstract formulation of \eqref{eq:1d-model}:
Find $u \in Q$ such that
 \begin{equation}\label{VFabstract}
b(v,u) = \<F, v\>, \ \text{for all} \ v \in V.
\end{equation}
where $Q$ and $V$ are separable Hilbert spaces and $F$ is a continuous linear functional on $V$.
We assume that the inner products $a_0(\cdot, \cdot)$ and $(\cdot, \cdot)_{{Q}}$ induce the  norms $|\cdot|_V =|\cdot| =a_0(\cdot, \cdot)^{1/2}$ and $\|\cdot\|_{Q}=\|\cdot\|=(\cdot, \cdot)_{Q}^{1/2}$. We denote the dual of $V$ by $V^*$ and the dual pairing on $V^* \times V$ by $\langle \cdot, \cdot \rangle$.  
We assume that $b(\cdot, \cdot)$ is a continuous bilinear form on $V\times Q$ satisfying
 the $\sup-\sup$ condition
 \begin{equation} \label{sup-sup_a}
\du{\sup}{u \in Q}{} \ \du {\sup} {v \in V}{} \ \frac {b(v, u)}{|v|\,\|u\|} =M <\infty, 
\end{equation} 
and the $\inf-\sup$ condition
 \begin{equation} \label{inf-sup_a}
 \du{\inf}{u \in Q}{} \ \du {\sup} {v \in V}{} \ \frac {b(v, u)}{|v|\,\|u\|} =m>0.
\end{equation}

With the form $b$, we associate the operators $\B:V\to {Q}$    defined by  
\[
(\B v,q)_{{Q}}=b(v, q) \,  \Forall  v \in V, q \in Q.
\]
We define $V_0$ to be the kernel of $\B$, i.e.,
\[ 
V_0 :=Ker(\B)= \{v \in V |\  \B v=0\}.
\]
Under assumptions \eqref{sup-sup_a}  and \eqref{inf-sup_a}, the operator $\B$ is a bounded surjective operator from $V$ to $Q$, and $V_0$ is a closed subspace of $V$. We will also assume that the  data $F \in V^*$ satisfies the {\it compatibility condition} 
\begin{equation}\label{eq:BBsuf}
\<F,v\> =0 \Forall v \in V_0=Ker(\B). 
\end{equation}
  
The following result describes the well posedness  of  \eqref{VFabstract} and can be used at the continuous  and discrete levels, see e.g.  \cite{A-B, B09, boffi-brezzi-demkowicz-duran-falk-fortin2006, boffi-brezzi-fortin}. 
\begin{prop} \label{prop:well4mixed} If the form $b(\cdot,\cdot)$ satisfies \eqref{sup-sup_a} and \eqref{inf-sup_a}, and the  data $F \in V^*$ satisfies the {\it compatibility condition}  \eqref{eq:BBsuf}, then  the problem \eqref{VFabstract} has  unique solution that depends continuously on the data $F$. 
\end{prop}
It is also known, see e.g., \cite{BM12, BQ15,BQ17, Dahmen-Welper-Cohen12} that, under the  {\it compatibility condition} \eqref{eq:BBsuf}, solving the mixed problem  \eqref{VFabstract} reduces to  solving a standard saddle point reformulation: Find $(w, u) \in V \times Q$ such that  
\begin{equation}\label{abstract:variational2}
\begin{array}{lclll}
a_0(w,v) & + & b( v, u) &= \langle F,v \rangle &\ \Forall  v \in V,\\
b(w,q) & & & =0   &\  \Forall  q \in Q.  
\end{array}
\end{equation}
In fact,  we have that $p$ is the unique solution of \eqref{VFabstract}  {\it if and only if} $(w=0 , p) $ solves  \eqref{abstract:variational2}, and the result remains valid if the form $a_0(\cdot, \cdot)$ in \eqref{abstract:variational2} is replaced by any other symmetric bilinear form  $a(\cdot, \cdot)$  on $V$ that leads to an equivalent norm on $V$. 




\section{Saddle point least squares discretization}\label{sec:abstract-discretization}

We will assume next that  $V$ and $Q$ are Hilbert spaces with norms and inner products as defined in Section  \ref{sec:ReviewLSPP}. 
Let $V_h\subset V$ and  $\M_h\subset Q$ be finite dimensional approximation spaces.
 We assume the following discrete $\inf-\sup$ condition holds for the pair of spaces $(V_h,\M_h)$:
 \begin{equation} \label{inf-sup_h}
\du{\inf}{p_h \in \M_h}{} \ \du {\sup} {v_h \in V_h}{} \ \frac {b(v_h, p_h)}{|v_h|\,\|p_h\|} =m_h>0.
\end{equation} 
As in the continuous case we  define 
\[
V_{h,0}:=\{v_h\in V_h\,|\, b(v_h,q_h)=0,\Forall q_h\in \M_h\},
\] 
 and $F_h \in V_h^*$ to be the restriction of $F$ to $V_h$, i.e.,   $\langle F_h, v_h \rangle:=\langle F, v_h \rangle$ for all $v_h \in V_h$. 
In the case $V_{h,0} \subset V_0$, the compatibility condition \eqref{eq:BBsuf} implies the discrete compatibility condition
\[
\langle F,v_h\rangle =0 \Forall v_h\in V_{h,0}.
\]
Hence, under assumption \eqref{inf-sup_h}, the PG  problem of finding $u_h\in \M_h$ such that 
\begin{equation}\label{discrete_var_form}
b(v_h, u_h)=\langle F,v_h\rangle,  \ v_h\in V_h
\end{equation}
has a unique solution. In general, we might not have  $V_{h,0}\subset V_0$. Consequently,  even though the continuous problem \eqref{VFabstract} is well posed, the discrete problem  \eqref{discrete_var_form} might not be  well-posed. However, if the form $b(\cdot, \cdot)$ satisfies \eqref{inf-sup_h}, then the problem of finding $(w_h,p_h) \in V_h\times \M_h$ satisfying  
\begin{equation}\label{discrete:variationalSPP}
\begin{array}{lclll}
a_0(w_h,v_h) & + & b( v_h, p_h) &= \langle f,v_h \rangle &\ \Forall  v_h \in V_h,\\
b(w_h,q_h) & & & =0   &\  \Forall  q_h \in \M_h, 
\end{array}
\end{equation} 
 does have a unique solution. 
We call  the component $u_h$  of the solution $(w_h,u_h)$ 
of \eqref{discrete:variationalSPP} the  {\it saddle point least squares} approximation of the solution $u$ of the original mixed problem \eqref{VFabstract}.

The following error estimate for $\|u-u_h\|$ was proved in \cite{BQ15}. 

 
\begin{theorem}\label{th:sharpEE} 
Let $b:V \times Q \to \R$  satisfy \eqref{sup-sup_a} and \eqref{inf-sup_a} and assume that   ${F}  \in V^*$  is given and satisfies \eqref{eq:BBsuf}. Assume that  $u$  is the  solution  of \eqref{VFabstract} and  $V_h \subset V$,  $ {\M}_h \subset  Q$ are  chosen such that the discrete $\inf-\sup$ condition   \eqref{inf-sup_h} holds. If  $\left (w_h, u_h \right )$ is the  solution  of \eqref{discrete:variationalSPP}, then the following error estimate holds:
\begin{equation}\label{eq:er4LS} 
 \frac 1 M |w_h| \leq \|u-u_h\| \leq  \frac{M}{m_h} \  \du{\inf}{q_h \in\M_h
}{}  \|u-q_h\|.
\end{equation} 
\end{theorem}
 
The considerations made so far in this section remain valid  if the form $a_0(\cdot, \cdot)$, as an inner product on $V_h$, is replaced by another inner product $a(\cdot, \cdot)$ which gives rise to an equivalent  norm on $V_h$. 


\section{Discretization with $C^0-P^1$ trial space for the 1D Convection reaction problem}\label{sec:1d-lin-discrete}
In this section we   review  standard finite element discretizations of \eqref{eq:1d-model} and emphasize the ways the corresponding linear system relate. The concepts presented in this section are focused on uniform mesh discretization, but most of the results can be easily  extended to non-uniform meshes. 

We  divide the interval $[0,1]$ into $n$  equal length subintervals,  using the nodes $0=x_0<x_1<\cdots < x_n=1$ and denote  $h:=x_j - x_{j-1}, j=1, 2, \cdots, n$. For the above uniform distributed notes on $[0, 1]$, we define  the corresponding   discrete space  $\M_h$  as  the subspace of $Q = H^1_0(0,1)$, given by
 \[ 
 \M_h = \{ v_h \in V \mid v_h \text{ is linear on each } [x_j, x_{j + 1}]\},
 \]
  i.e., $\M_h$ is the space of all {\it piecewise linear continuous functions} with respect to the given nodes, that are zero at $x=0$ and $x=1$.  We consider the nodal basis $\{ \varphi_j\}_{j = 1}^{n-1} \subset V_h$ with the standard defining property $\varphi_i(x_j ) = \delta_{ij}$. 

\subsection{Standard Linear discretization} \label{sec:LinP1}
We couple the above discrete trial space with a discrete  test space $V_h:=\M_h$.  
 Thus, the standard (linear) discrete variational formulation of \eqref{VF1d} is: Find $u_h \in \M_h$ such that
 \begin{equation}\label{dVF}
b(v_h, u_h) = (f, v_h), \ \text{for all} \ v_h \in V_h.
\end{equation}
We look for  $u_h \in V_h$ with  the nodal basis expansion
 \[
 u_h := \sum_{i=1}^{n-1} u_i \varphi_i, \ \text{where} \ \ u_i=u_h(x_i).
 \]
If we consider the test functions $v_h=\varphi_j, j=1,2,\cdots,n-1$ in \eqref{dVF}, we  obtain  the following linear system 
 \begin{equation}\label{1d-LS}
 \left (\frac{\varepsilon}{h}  S+ C \right )\, U = F, 
\end{equation}

where \(U,F\in\R^{n-1}\) and \(S, C \in\R^{(n-1)\times(n-1)}\) with:
 \[
 U:=\begin{bmatrix}u_1\\u_2\\\vdots\\u_{n-1}\end{bmatrix},\quad F:=\begin{bmatrix}(f,\varphi_1)\\ (f,\varphi_2)\\\vdots \\ (f,\varphi_{n-1})\end{bmatrix}, \text{and} 
\]
 \[
 S:=\begin{bmatrix}2&-1\\-1&2&-1\\&\ddots&\ddots&\ddots\\&&-1&2&-1\\&&&-1&2\end{bmatrix},\quad 
 C:=\frac{1}{2}\begin{bmatrix}0&1\\-1&0&1\\&\ddots&\ddots&\ddots\\&&-1&0&1\\&&&-1&0\end{bmatrix}.\] 
Note that, letting   $\varepsilon \to 0$  in \eqref{VF1d} we obtain the {\it simplified problem}: \\
Find $w \in H_0^1(0,1)$ such that
\begin{equation} \label{VF1d-simplified-w}
(w',v) = (f, v), \ \text{for all} \ v \in V. 
\end{equation}

The problem \eqref{VF1d-simplified-w} has unique solution,  if and only if $\int_0^1  f(x)  \, dx=0$.    For the case $\int_0^1  f(x)  \, dx \neq 0$ we can consider  the  {\it reduced problem}: \\ Find $w \in H^1(0,1)$ such that
 \begin{equation}\label{VF1d-reduced}
w'(x) = f(x)  \ \text{for all} \ x \in (0, 1), \text{and} \ w(0)=0. 
\end{equation}
with the unique solution $w(x) = \int_0^x  f(x)  \, dx$. 

The corresponding finite element discretization of the {\it simplified problem} \eqref{VF1d-simplified-w} leads  to find  $w_h := \sum_{i=1}^{n-1} u_i \varphi_i$, where
\begin{equation}\label{lin-reduced}
C\, U = F.
\end{equation}
It is interesting to note that, even though \eqref{VF1d-simplified-w} is not well posed in general, the  system \eqref{lin-reduced} decouples into two independent systems, and at least for  $n=2m+1$,  it has unique solution. Indeed, by defining  $u_0=u_n=0$, then for the case $n=2m+1$ we get 
\begin{equation}\label{sys-even}
\begin{cases} u_2 -u_0 &=2(f, \varphi_1) \\  u_4 -u_2&=2(f, \varphi_3)\\ \vdots \\ 
u_{2m} -u_{2m-2}&=2(f, \varphi_{2m-1}), 
\end{cases}
\end{equation}
and
 \begin{equation}\label{sys-odd}
\begin{cases} u_3 -u_1 &=2(f, \varphi_2) \\  u_5 -u_3&=2(f, \varphi_4)\\ \vdots \\ 
u_{2m+1} -u_{2m-1}&=2(f, \varphi_{2m}).
\end{cases}
\end{equation}
In this case the  systems \eqref{sys-even} and \eqref{sys-odd} have  unique solutions, and can be solved forward and backward  respectively, to get 

 \begin{equation}\label{sys-sol-odd}
\begin{cases} u_{2k} &=2 \sum_{j=1}^k (f, \varphi_{2j-1}), \ k=1,2,\cdots,m \\ 
 u_{2m-2k+1} &=-2 \sum_{j=1}^{k} (f, \varphi_{2m-2j+2}), \ k=1,2, \cdots,m
\end{cases}
\end{equation}
For $f=1$ on $[0, 1]$, we have $(f,\varphi_i) =h$ for all $i=1,2, \cdots, 2m$, and 
 \begin{equation}\label{sys-sol-odd-f1}
\begin{cases} u_{2k} &=2kh=x_{2k}, \ k=1,2,\cdots,m \\ 
 u_{2m-2k+1} &=- 2kh =x_{2m-2k+1}-1, \ k=1,2, \cdots,m.
\end{cases}
\end{equation}
Thus, the even components interpolate the  solution of the function $x$  and the odd components interpolate the function $x-1$.
The combined solution leads to a very oscillatory behavior when $n\to \infty$.  For $\varepsilon/h < < 1$ (a good threshold  is $\varepsilon/h \leq 10^{-4}$ ) the solution of \eqref{dVF} is very close to the solution of the simplified  system \eqref{lin-reduced}, and a similar oscillatory behavior  is observed for the linear  finite element solution of \eqref{dVF} when using an odd number of subintervals $n$, see Fig.1. We note  that, for an arbitrary (smooth) $f$,  the even components $\{u_{2k}\}$, approximate  $w(x)$ the solution of the initial value problem (IVP)  \eqref{VF1d-reduced}, and the  the odd components approximate  the function $\theta(x)=w(x)- \int_0^1 f(x)\, dx$, see Fig.1 and Fig.5. This can be justified by noticing that if we replace in \eqref{sys-even} the values $(f,\varphi_i)$  by   $h\, f(x_i)$ - the corresponding trapezoid rule approximation of the integral, the solution of the modified system coincides with the  mid-point method approximation (on the even nodes, $h\to 2h$) of the IVP \eqref{VF1d-reduced}. \\

 \parbox{2.3in}{
\begin{center}
\includegraphics[width=2.3in]{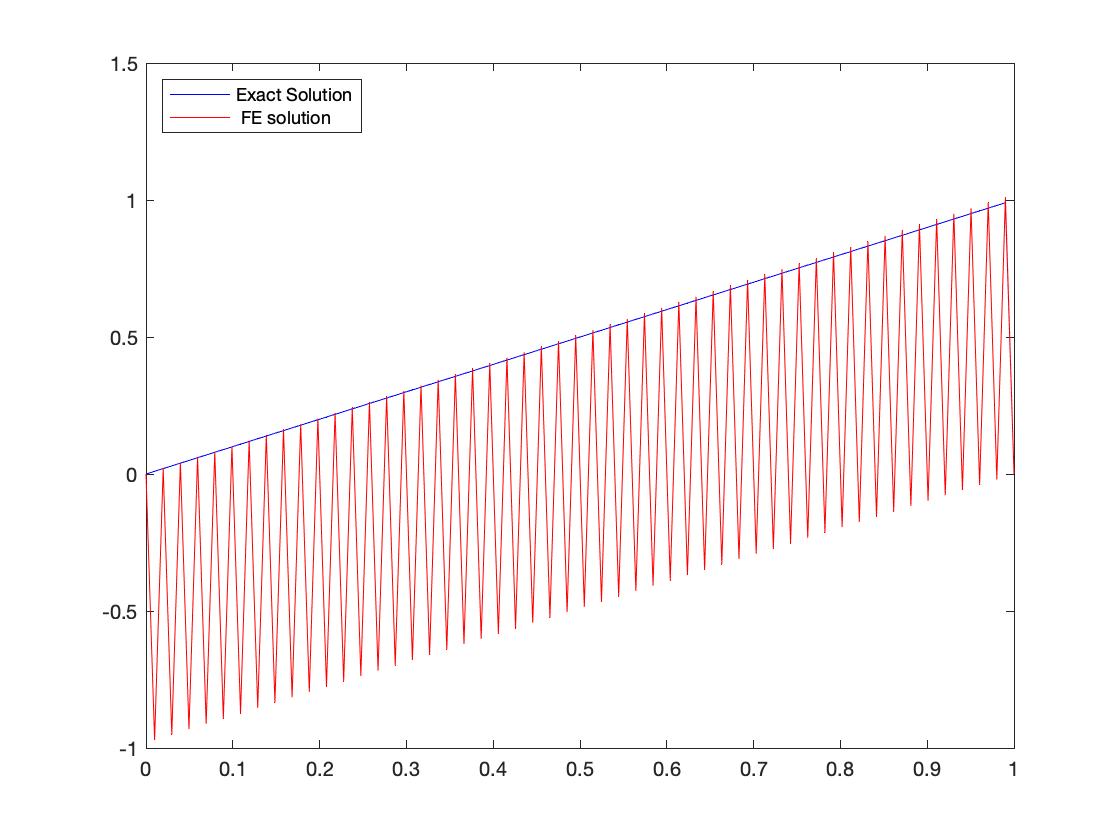}\\
{Fig.1: $f=1, n=101, \varepsilon=10^{-6}$} 
\end{center}
}
\parbox{2.3in}{
\begin{center}
\includegraphics[width=2.3in]{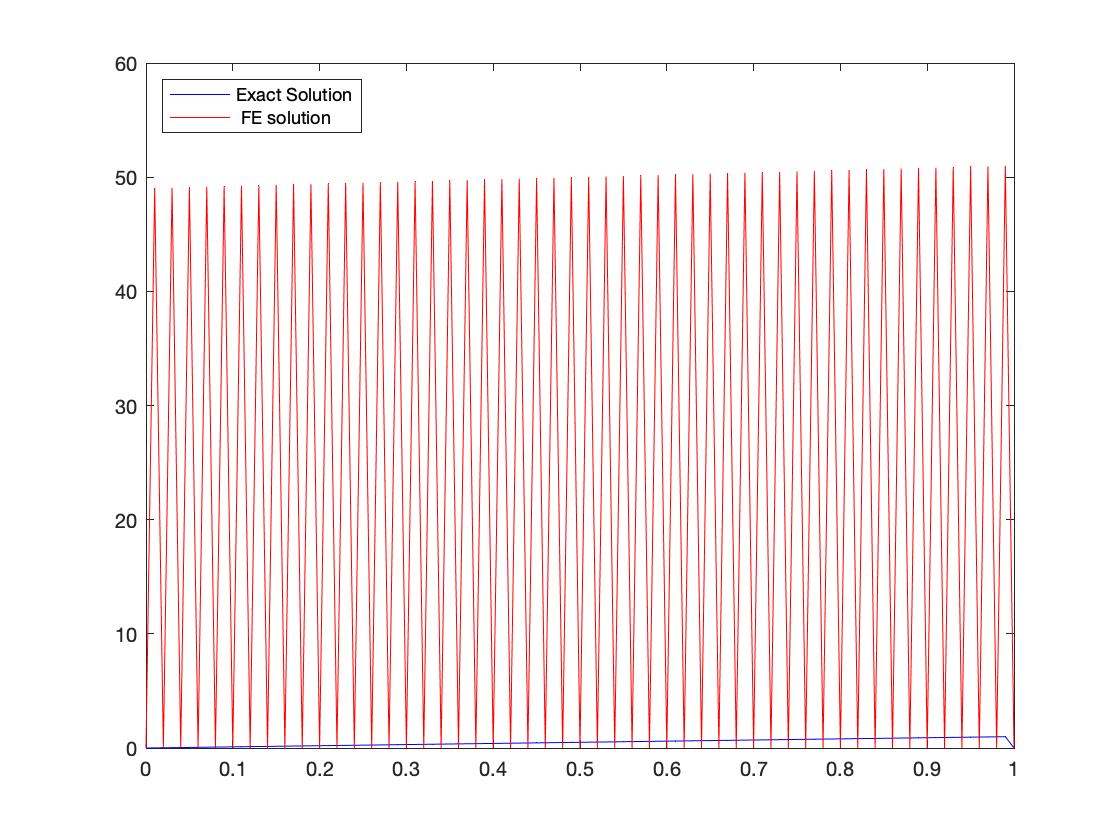}\\
{Fig.2: $f=1, n=102$,  $\varepsilon=10^{-6}$} 
\end{center}
}

 
 \parbox{2.3in}{
\begin{center}
\includegraphics[width=2.3in]{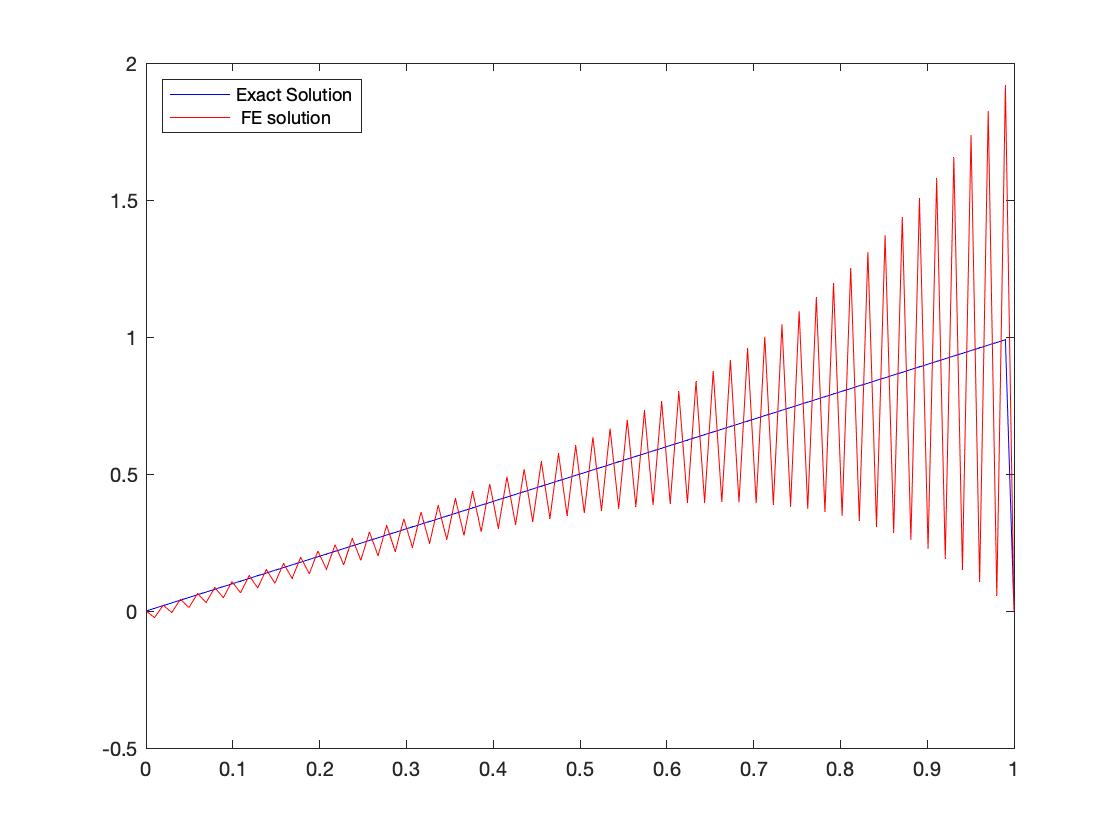}\\
{Fig.3:  $f=1, n=101, \varepsilon=10^{-4}$} 
\end{center}
}
\parbox{2.3in}{
\begin{center}
\includegraphics[width=2.3in]{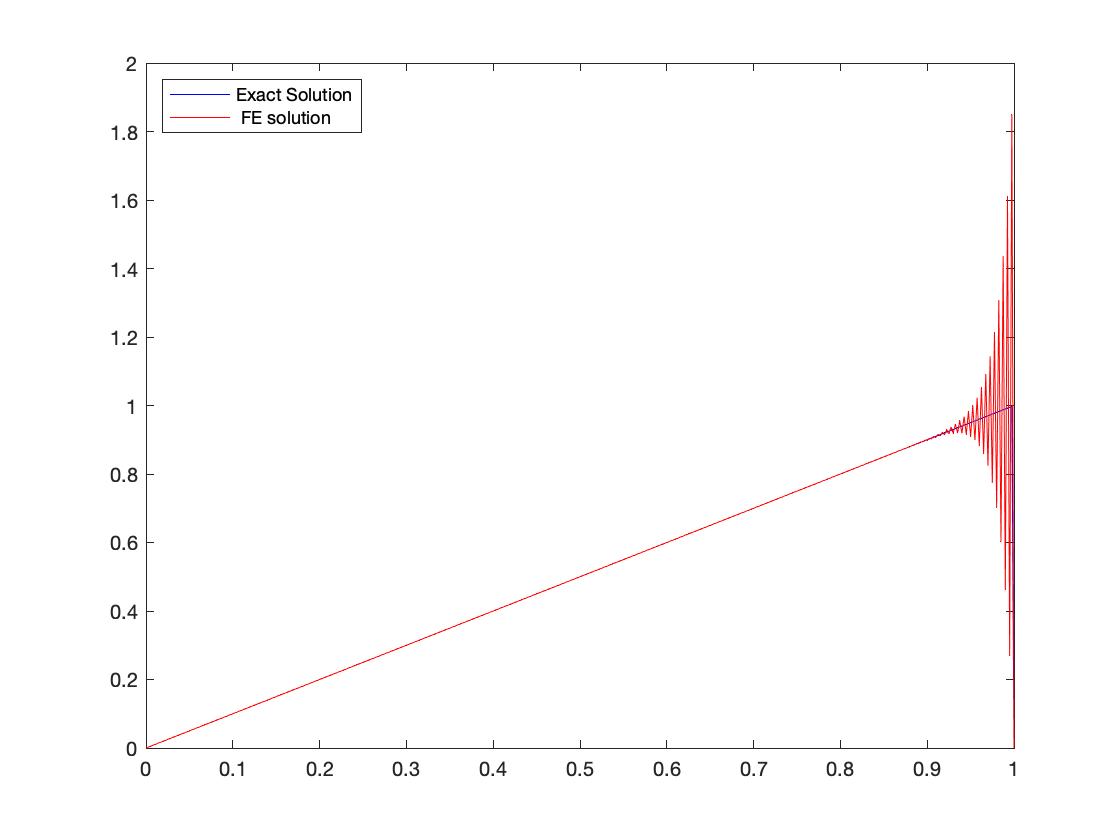}\\
{Fig.4: $f=1, n=400, \varepsilon=10^{-4}$} 
\end{center}
}
\vspace{0.1in}

 \parbox{2.3in}{
\begin{center}
\includegraphics[width=2.3in]{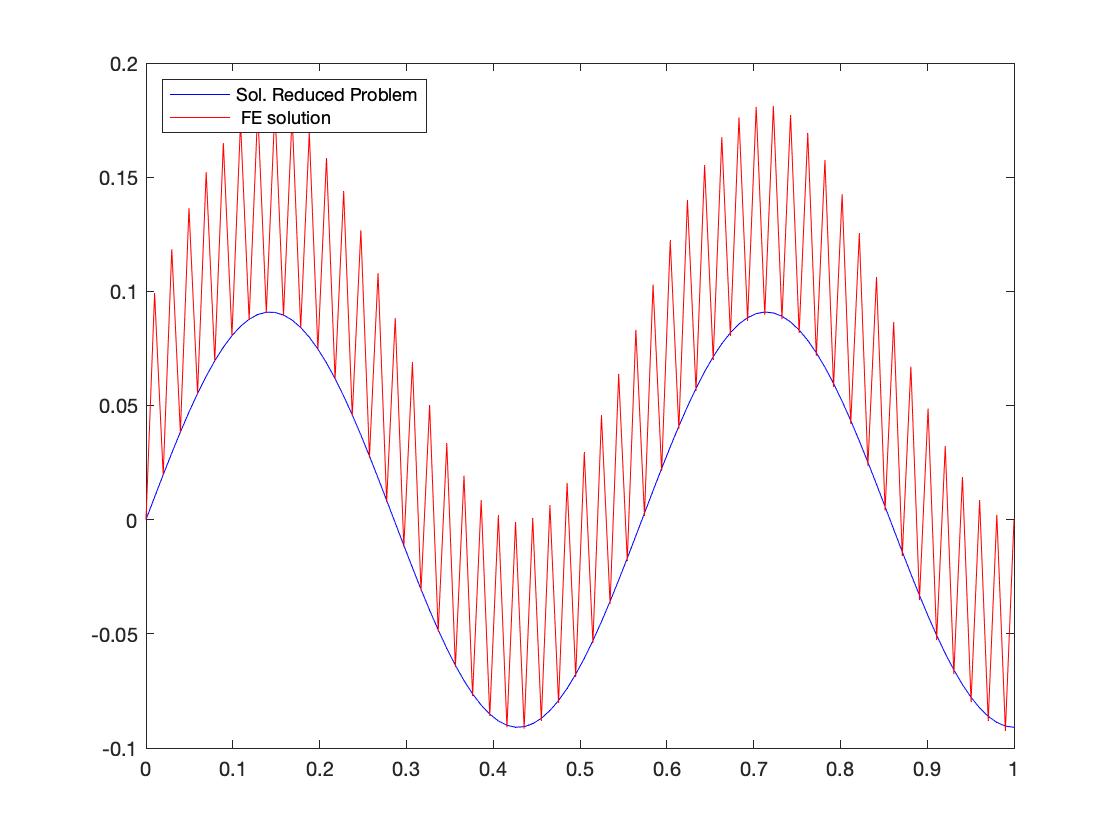}\\
{Fig.5:$f=\cos(\frac{7\pi}{2} x)$,  $n=101,  \ \ \ \ \varepsilon=10^{-6}$} 
\end{center}
}
\parbox{2.3in}{
\begin{center}
\includegraphics[width=2.3in]{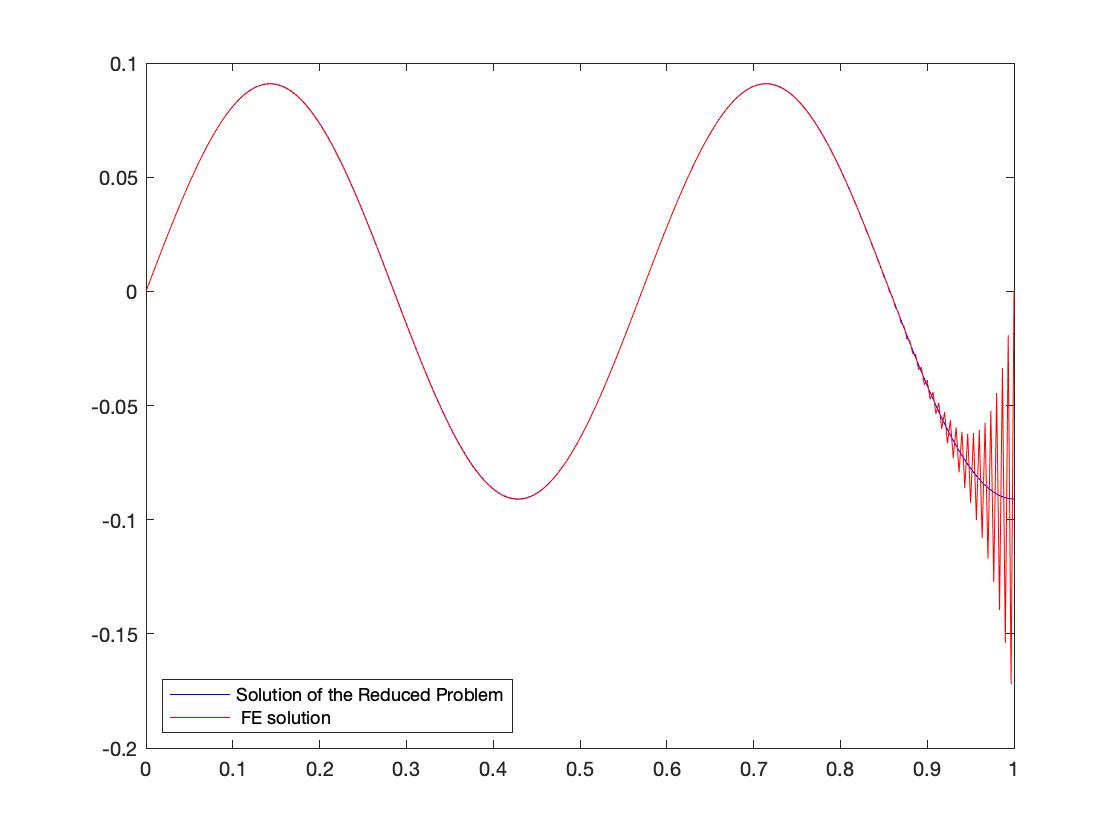}\\
{$f=\cos(\frac{7\pi}{2} x)$, $ n=300, \varepsilon=10^{-4}$} 
\end{center}
}
\vspace{0.1in}

Similarly,  the  solution of the modified  system \eqref{sys-odd} (obtained by replacing $(f,\varphi_i)$  with  $h\, f(x_i)$) coincides with the mid-point method approximation (on odd nodes)  of  the IVP \\
\begin{equation}\label{VF1d-reduced2}
\theta'(x) = f(x)  \ \text{for all} \ x \in (0, 1), \text{and} \ \theta(1)=0.
\end{equation}
The solution of \eqref{VF1d-reduced2} is  $\theta(x)= -\int_x^1 f(s)\, ds$. Thus,  $\theta(x)=w(x)- \int_0^1 f(x)\, dx$.
For the case $n=2m$, the system \eqref{sys-even} is the same, but since $u_0=u_{2m} =0$, the system might not have a solution.  In addition, the second system   \eqref{sys-odd} (with  the last equation  removed)  is undetermined and could have infinitely many solutions. The discretization of \eqref{dVF}  is still very oscillatory in this case, see Fig.2. As the ratio $\varepsilon/h \to 1$, from  numerical  tests, we  note  that the linear finite element solution of \eqref{dVF}  oscillates between two curves (that depend on $h$ and are independent of the parity of the number of nodes),  and approximate well the graph of $w$ on intervals $[0, \, \alpha(h)]$ with $\alpha(h)\to 1$  as $h$ gets closer and closer to $\epsilon$, see Fig.3, Fig.4, and Fig.6.

The behavior of the standard linear finite element approximation  motivates the need for other methods, including  {\it saddle point   least square} or   {\it Petrov-Galerkin} methods.

\subsection{$(P^1-P^2)$-SPLS discretization} \label{sec:MMD}
For improving the stability  and approximability of the finite element approximation a  {\it  saddle point least square}  (SPLS) method can be used, see e.g., \cite{Dahmen-Welper-Cohen12, dem-fuh-heu-tia19, BM12}.  The SPLS method for solving \eqref{VF1d} is: Find $(w, u) \in V \times Q$ such that 
\begin{equation}\label{SPLS4model2}
\begin{array}{lclll}
a_0(w,v) & + & b( v, u) &= (f,v ) &\ \Forall  v \in V,\\
b(w,q) & & & =0   &\  \Forall  q \in Q,   
\end{array}
\end{equation}
where $V=Q= H^1_0(0,1)$, with possible  different type of norms, and \\  $b(v, u) =\varepsilon\, a_0(u, v)+(u',v)= \varepsilon\, (u', v')+(u',v)$. 

For the discretization of \eqref{SPLS4model2}  we choose finite element space
$\M_h \subset Q$ and $V_h \subset V$ and solve the discrete problem: 
Find $(w_h, u_h) \in V_h \times \M_h$ such that 
\begin{equation}\label{SPLS4model-h}
\begin{array}{lclll}
a_0(w_h,v_h) & + & b( v_h, u_h) &= (f,v_h ) &\ \Forall  v_h \in V_h,\\
b(w_h,q_h) & & & =0   &\  \Forall  q_h \in \M_h.   
\end{array}
\end{equation}


Analysis and numerical results  for finite element  test and trail spaces of various degree polynomial were done in  \cite{dem-fuh-heu-tia19}. We present next some numerical observations for  $\M_h= C^0-P^1:= span\{ \varphi_j\}_{j = 1}^{n-1}$, with $\varphi_j$'s  the standard linear nodal functions  and $V_h=C^0-P^2$  on the given uniformly distributed nodes on $[0, 1]$, to show the  improvement from the standard linear discretization.  
The presence of non-phisical oscillation is diminished, and the  errors are better for the SPLS discretization, see Table 1 and Table 2.  

While for  $\int_0^1f(x)\, dx=0$ there is no much difference in the solution behaviour for the two methods, 
for  $\int_0^1  f(x) \, dx\neq 0$, numerical tests  showed  an essential improuvement for the SPLS solution. Inside the interval $[3h, 1-3h]$ the SPLS solution $u_h$, approximates the  shift by a constant  of  the solution $u$ of the original problem \eqref{VF1d},   see Fig.7-Fig.10. The oscillations appear only at the ends of the interval. 
The behavior can be explained by similar arguments presented in Section \ref{sec:LinP1} as follows: 
The {\it simplified} problem, obtained from \eqref{SPLS4model2} by letting $\varepsilon \to 0$, is not well posed when $\int_0^1  f(x)  \, dx\neq 0$. However, the {\it simplified} linear system obtained from \eqref{SPLS4model-h}  by letting $\varepsilon \to 0$, i.e. find $(w_h, u_h) \in V_h \times \M_h$ such that 
\begin{equation}\label{SPLS4model-h-R}
\begin{array}{lclll}
(w'_h,v'_h) & + & (u'_h, v_h) &= (f,v_h ) &\ \Forall  v_h \in V_h,\\
(w_h,q'_h) & & & =0   &\  \Forall  q_h \in \M_h,   
\end{array}
\end{equation}
 has unique solution, because a discrete $\inf-\sup$ condition can be demonstrated using a specific choice of norms.    Numerical tests  (for $\varepsilon \leq 10^{-3}$) show that the solution of the simplified system \eqref{SPLS4model-h-R} approximates  (when $h\to0$) the function  $\frac{1}{2} (w(x) +\theta(x))$ where $w, \theta$, are  the solution of the reduced problems \eqref{VF1d-reduced} and \eqref{VF1d-reduced2}. A similar type of  oscillations (depending only on $h$) towards the ends of $[0, 1]$ are still presented. For example, for $f=1$  the solution of \eqref{SPLS4model-h-R} with $n=101$, is  close to $x-1/2$, see  Fig.7. 
For  $\varepsilon/h \leq 10^{-4}$ the solution of \eqref{SPLS4model-h} is close to the solution of   \eqref{SPLS4model-h-R}.  However, as $10^{-4}<\varepsilon/h  \to 1$, the  solution of  \eqref{SPLS4model-h} is decreasing  the size of the shifting constant and approximates  $u$ (rather than ${1}/{2} (w(x) +\theta(x))$).  Similar oscillations are still present, but only outside of the interval $[3h, 1-3h]$. 


 \parbox{2.5in}{
\begin{center}
\includegraphics[width=2.5in]{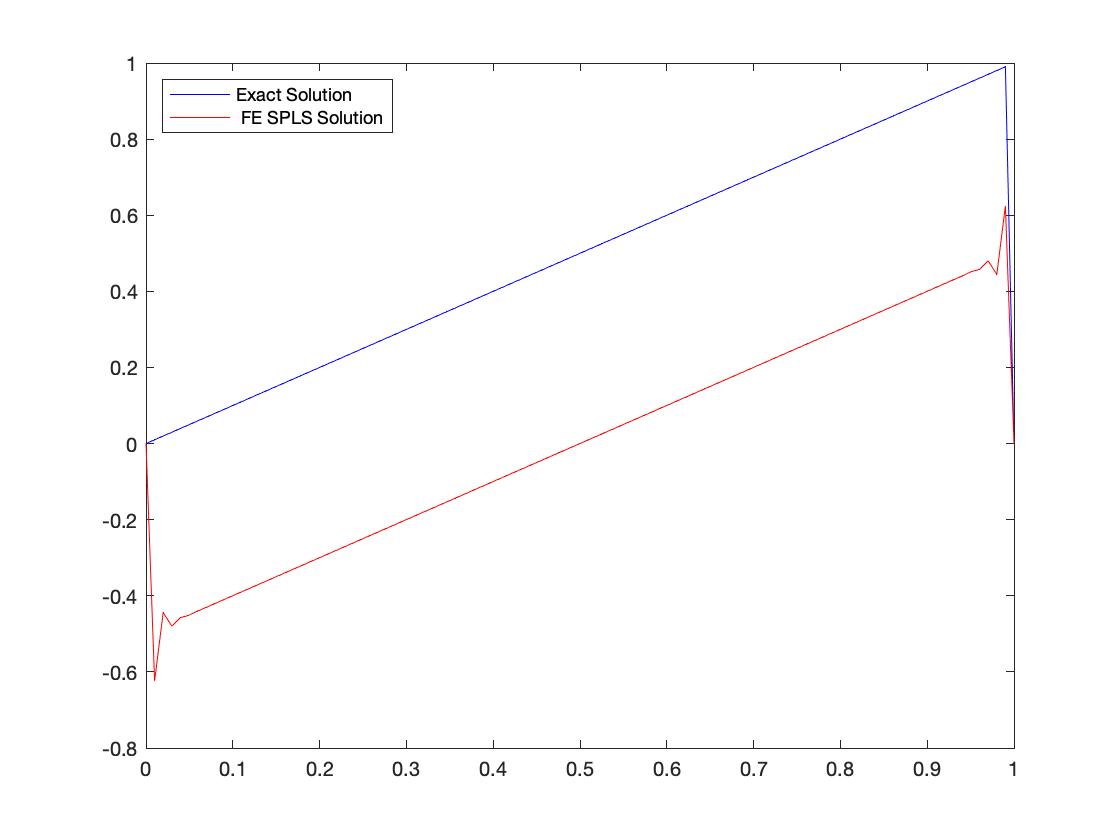}\\
{Fig.7:  $f=1, n=101, \varepsilon=10^{-6}$} 
\end{center}
}
\parbox{2.5in}{
\begin{center}
\includegraphics[width=2.5in]{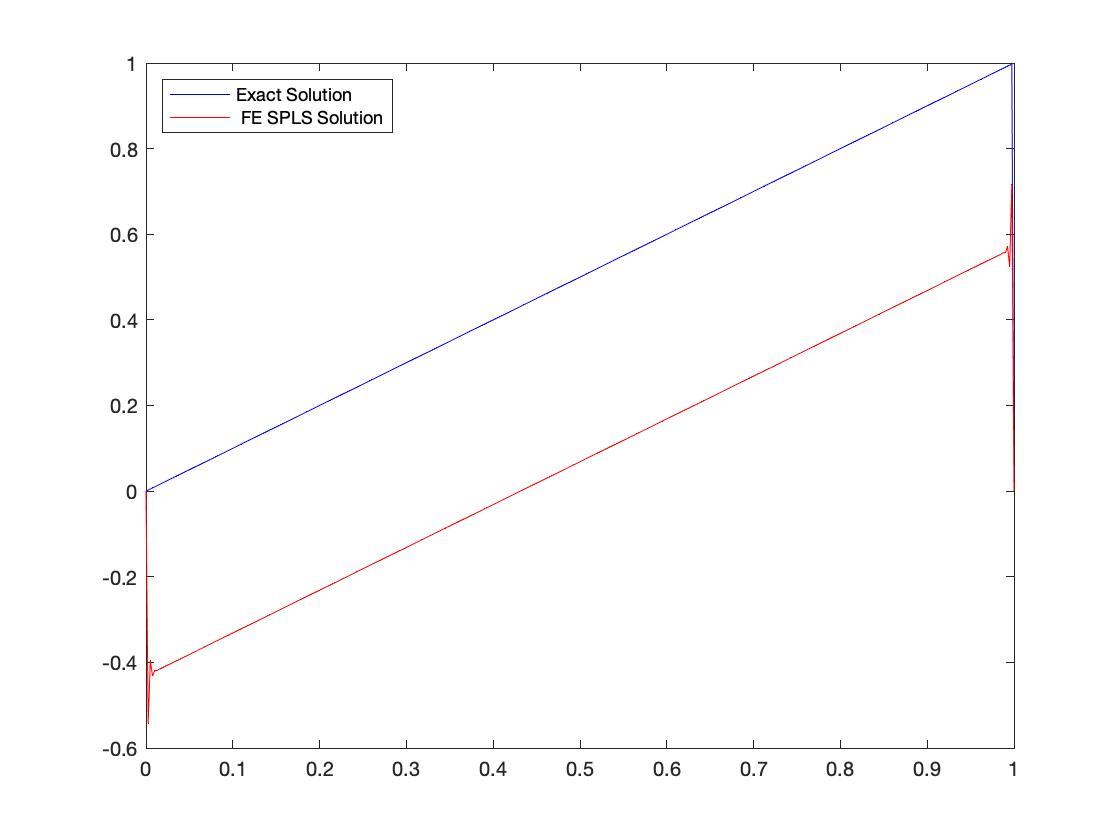}\\
{Fig.8: $f=1, n=400, \varepsilon=10^{-4}$} 
\end{center}
}
\vspace{0.1in}

 \parbox{2.5in}{
\begin{center}
\includegraphics[width=2.5in]{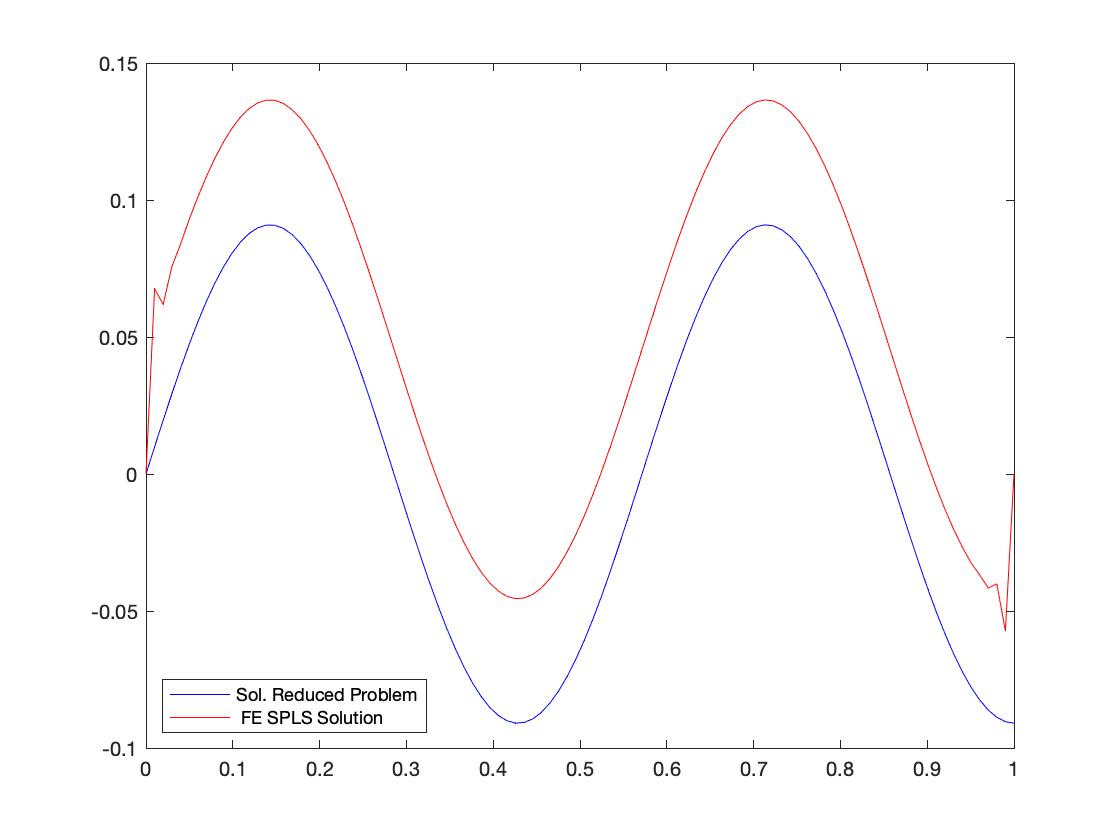}\\
{Fig.9 $f=\cos(\frac{\pi}{2} x), n=101, \varepsilon=10^{-6}$} 
\end{center}
}
\parbox{2.5in}{
\begin{center}
\includegraphics[width=2.5in]{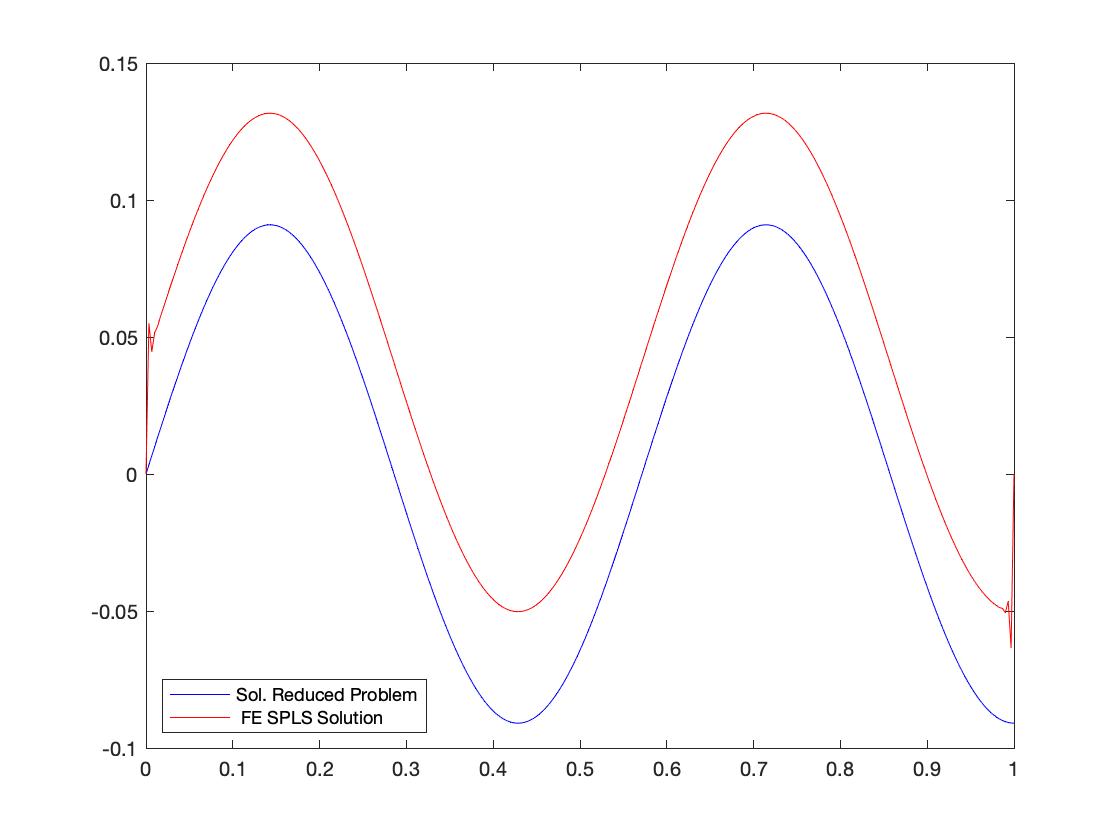}\\
{Fig.10 $n=300, \varepsilon=10^{-4}$} 
\end{center}
}
\vspace{0.1in}

\subsection{Petrov Galerkin (PG) with bubble enriched test space $V_h$}\label{ssec:PGbubbles}

We consider   $b(v,u) :=\varepsilon\, a_0(u, v)+(u',v) $  for all $ u, v \in V:=H^{1}_0(0,1)$. 
The second equation in \eqref{SPLS4model2} implies $w=0$, and the SPLS problem reduces to:  Find $u \in Q$ such that 
\begin{equation}\label{PG4model}
b( v, u) = (f,v ) \ \Forall  v \in V,\\
\end{equation}
which is a Petrov-Galerkin method for solving \eqref{eq:1d-model}.

One of the well known  Petrov-Galerkin discretization of the model problem \eqref{PG4model} with $\M_h= span\{ \varphi_j\}_{j = 1}^{n-1}$ consists of modifying the test space such that diffusion is created from the reaction therm.  This is also known as an {\it up-winding}  finite element scheme, see Sectioin 2.2 in \cite{roos-stynes-tobiska-96}. We define the test space $V_h$, by introducing first  a  bubble function for each interval $[x_{i-1}, x_i], i=1,2, \cdots, n$:
 \[
B_i:= 4\, \varphi_{i-1}\, \varphi_i, \ \ i=1,2, \cdots, n,
\] 
which is supported in $[x_{i-1}, x_i]$. The discrete test space $V_h$ is 
\[
V_h:= span \{ \varphi_j  + B_{j}-B_{j+1}\}_{j = 1}^{n-1}. 
\]
We note that both $\M_h$ and $V_h$ have dimension $n-1$ and, in a more general approach the test functions can be defined using  up-winding parameters $\sigma_i >0$ to get $V_h:= span \{ \varphi_j  + \sigma_i( B_{j}-B_{j+1})\} _{j = 1}^{n-1}$. 

\subsubsection{Variational formulation  and matrices}
The Petrov Galerkin discretization for 
\eqref{eq:1d-model} is: Find $u_h \in \M_h$ such that 
\begin{equation}\label{PG4model-h}
b( v_h, u_h) = (f,v_h) \ \Forall  v_h \in V_h. 
\end{equation}
We look for 
\[
u_h= \sum_{j=1}^{n-1} \alpha_j \varphi_j,
\]
and consider a generic test function 
\[
v_h= \sum_{i=1}^{n-1} \beta_i \varphi_i + \sum_{i=1}^{n-1}  \beta_i (B_i - B_{i+1}) = \sum_{i=1}^{n-1} \beta_i \varphi_i + \sum_{i=1}^{n}  (\beta_i - \beta_{i-1}) B_{i},
\]
where, we define $\beta_0=\beta_n=0$. Denoting,
\[
B_h:=\sum_{i=1}^{n}  (\beta_i - \beta_{i-1}) B_{i},  \ \text{and } \  w_h:=  \sum_{i=1}^{n-1} \beta_i \varphi_i ,
\]
we have 
\[
v_h=w_h + B_h.
\]
We note that for a generic bubble function $B$  with support $[a, b]$ we have
\[
B:= \frac{4}{(b-a)^2} (x-a) (b-x), \ \text{with}  \ a<b,\  \text{and}
\]
\begin{equation} \label{eq:B-prop}
\int_a^b B(x)\, dx  = \frac{2 (b-a)}{3}, \ 
\int_a^b B'\, dx  =0, \ 
\int_a^b (B')^2\, dx  =\frac{16}{3 (b-a)}.
\end{equation}
Using the above formulas, the fact that $u'_h, w'_h $  are constant  on each of  the intervals $[x_{i-1}, x_i]$,  and that $w'_h= \frac{\beta_i -\beta_{i-1} }{h}$ on $[x_{i-1}, x_i]$,  we obtain
\[
(u'_h, B_h) = \sum_{i=1}^{n}\int_{x_{i-1}}^{x_i}  u'_h (\beta_i - \beta_{i-1}) B_{i}=
 \sum_{i=1}^{n} u'_h \,  w'_h \int_{x_{i-1}}^{x_i}  B_{i} =\frac{2h}{3}  \sum_{i=1}^{n} \int_{x_{i-1}}^{x_i}  u'_h  w'_h. 
\]
Thus
\begin{equation} \label{eq:upBh}
(u'_h, B_h) =\frac{2h}{3}  (u'_h, w'_h), \ \text{where} \  v_h=w_h + B_h.
\end{equation}

In addition, 
\[
(u'_h, B'_i) =0 \ \text{for all} \ i=1, 2, \cdots,  n, \text{hence} 
\]
\begin{equation} \label{eq:upBph}
 (u'_h, B'_h) =0,  \text{for all} \  u_h \in \M_h, v_h=w_h + B_h \in V_h.
\end{equation}
From  \eqref{eq:upBh} and  \eqref{eq:upBph}, for any $u_h \in \M_h, v_h=w_h + B_h \in V_h$ we get
\begin{equation} \label{eq:bPG}
 b(v_h, u_h) = \left (\varepsilon + \frac{2h}{3}\right )  (u'_h, w'_h) +  (u'_h, w_h).
\end{equation}
Thus, adding the bubble part to the test space leads to the extra diffusion term $  \frac{2h}{3}  (u'_h, w'_h)$ with  $\frac{2h}{3} >0$  matching the sign of the coefficient of $u'$ in  \eqref{eq:1d-model}. It is also interesting to note  that  only the linear part of $v_h$ appears in expression of $ b(v_h, u_h)$.  The  functional  $v_h \to (f, v_h)$ can  be also viewed as functional only of the linear part $w_h$. Indeed, using the splitting $v_h=w_h + B_h $  and that $B_h:=\sum_{i=1}^{n}  (\beta_i - \beta_{i-1}) B_{i}$ we get
\[
(f, v_h) =(f, w_h) + (f, \sum_{i=1}^{n}  h w'_h B_i) =(f, w_h) +h\,  (f, w'_h  \sum_{i=1}^{n}   B_i). 
\]
The variational formulation of the up-winding  Petrov-Galerkin method can be reformulated as: Find $u_h \in \M_h$ such that 
\begin{equation}\label{PG4model-hR}
 \left (\varepsilon + \frac{2h}{3}\right )  (u'_h, w'_h) +  (u'_h, w_h) = (f, w_h )  + h\,  (f, w'_h  \sum_{i=1}^{n}   B_i),  w_h \in M_h. 
\end{equation}
 The reformulation allows for  a new error analysis  using an optimal test norm, see e.g. \cite{BHJ, BHJ22, BHO}, and for comparison with the  known {\it stream-line diffusion} (SD) method  of discretization  that is reviewed in the next section. 
 
 For the analysis of the method, using \eqref{eq:upBph} and the last part of \eqref{eq:B-prop},  we  note that for any $v_h=w_h + B_h \in V_h$ we have 
 \[
 \begin{aligned}
 (v'_h, v'_h) & =(w'_h + B'_h, w'_h + B'_h) = (w'_h, w'_h) + ( B'_h, B'_h) =\\ 
 & =(w'_h, w'_h) +  \sum_{i=1}^{n} (\beta_i-\beta_{i-1})^2 (B'_i, B'_i)=\\
 & =(w'_h, w'_h)   + \frac{16h}{3}  \sum_{i=1}^{n} \left (\frac{\beta_i-\beta_{i-1}}{h} \right )^2 = \\
 & =(w'_h, w'_h)   +  \frac{16}{3}  \sum_{i=1}^{n} \left (\int_{x_{i-1}}^{x_i} (w'_h)^2 \right )^2 =(w'_h, w'_h)   +  \frac{16}{3} (w'_h, w'_h). 
 \end{aligned}
 \]
 
Consequently,
\begin{equation} \label{eq:vhwh}
|v_h|^2 = \frac{19}{3} |w_h|^2. 
\end{equation}
Using the reformulation \eqref{PG4model-hR} the  linear system  to be solved is 
\begin{equation}\label{1d-PG-ls}
\left ( \left (\frac{\varepsilon}{h} +  \frac{2h}{3} \right ) S+ C \right )\, U = F_{PG}, 
\end{equation}
where \(U,F_{PG}\in\R^{n-1}\)  with:
 \[
 U:=\begin{bmatrix}u_1\\u_2\\\vdots\\u_{n-1}\end{bmatrix},\quad F_{PG}:= \begin{bmatrix}(f,\varphi_1)\\ (f,\varphi_2)\\\vdots \\ (f,\varphi_{n-1})\end{bmatrix} + 
 \begin{bmatrix} (f, B_1 -B_2) 
 \\  (f, B_2-B_3)\\ \vdots \\  (f, B_{n-1} -B_n) \end{bmatrix}, 
\]
and $S, C$ are the matrices defined at the beginning of this section.
Numerical tests, show that this method  does not lead to any kind of non-physical oscillations. 

\subsection{Stream line diffusion (SD)  discretization} \label{ssec:SD} 
The classical way to introduce this method can be found in e.g., \cite{hughes-brooks79, brezzi-marini-russo98}. For our model problem, we  present a simple way to introduce and relate the method with the up-winding PG method. 
We take  $\M_h= V_h= span\{ \varphi_j\}_{j = 1}^{n-1}$ and consider the  {\it stream line diffusion method} for solving  
\eqref{eq:1d-model}: Find $u_h \in \M_h$ such that 
\begin{equation}\label{SD4model-h}
b_{sd} (w_h, u_h) = F_{sd}(w_h) \ \Forall  w_h \in V_h, 
\end{equation} 
where 
\[
b_{sd} (w_h, u_h):= \varepsilon\, (u'_h, w'_h)+(u'_h,w_h) +  \sum_{i=1}^{n} \delta_i \int_{x_{i-1}}^{x_i} u'_h w'_h
\]
with $\delta_i >0$ weight parameters, and 
\[
F_{sd}(w_h) := (f, w_h) + \sum_{i=1}^{n} \delta_i \int_{x_{i-1}}^{x_i} f(x)\,  w'_h\, dx. 
\]
In practice $\delta_i$'s  are chosen proportional with $x_i-x_{i-1}=h$. 

For the choice 
\[
\delta_i = \frac{2h}{3}, \  i=1,2,\cdots, n,
\]
and arbitrary  $w_h, u_h \in \M_h =V_h$ the bilinear form $b_{sd}$ becomes 
\[
b_{sd} (w_h, u_h)=b(w_h,u_h) =\left (\varepsilon + \frac{2h}{3}\right )  (u'_h, w'_h) +  (u'_h, w_h),
\]
and the the corresponding right hand side  functional  $F_{sd}$ is 
\begin{equation}\label{eq:RHS-sd}
F_{sd}(w_h)  = (f, w_h) + \frac{2h}{3}(f,w'_h), \ w_h \in V_h. 
\end{equation} 
Thus, by choosing the appropriate  weights, the (up-winding)  PG  and SD discretization methods  lead to the the same stiffness matrix. Comparing the right hand sides of 
\eqref{PG4model-hR} and \eqref{eq:RHS-sd} we note that the two methods produce the same system (solution)  if and only if 
\begin{equation}\label{eq:RHS=}
(f, w'_h  \sum_{i=1}^{n}   B_i)=  \frac{2}{3}(f,w'_h), \ \text{for all} \ w_h \in V_h.
\end{equation} 
This is a  feasible condition, as 
\[
 \int_0^1 \sum_{i=1}^{n}   B_i = n\, \frac{2h}{3}= \frac{2}{3}.
 \]
In fact, the condition \eqref{eq:RHS=}  is satisfied for $f=1$. In this case, both sides of \eqref{eq:RHS=}   are zero. 
In general, we expect that, for certain  error norms, the PG to perform better. 
It is known, \cite{bartels16,quarteroni-sacco-saleri07,roos-stynes-tobiska-96} that the error estimate for the SD method is defined using a special SD-norm that, in the one dimensional case with  same  weights $\delta_i =\delta$, becomes
\[
\|v\|^2_{sd} = \varepsilon |v|^2 + \delta  |v|^2. 
\]
For a fair comparison with the PG method we take $\delta = \frac{2h}{3}$.  For the continuous solution $u$ of 
\eqref{eq:1d-model} and the discrete solution $u_h$ of \eqref{SD4model-h}, we have 

\begin{equation}\label{eq:EE-sd}
\|u-u_h\|_{sd} \leq  c_{sd}\,  h^{3/2} \|u''\|.
\end{equation}

For comparison of the implementation of the two methods we can compare also the load vector $F_{PG}$ defined above with  the load vector for the SD method:
\[
F_{SD}:= \begin{bmatrix}(f,\varphi_1)\\ (f,\varphi_2)\\\vdots \\ (f,\varphi_{n-1})\end{bmatrix} + 
 \frac{2h}{3} \begin{bmatrix} (f, \varphi'_1) 
 \\  (f, \varphi'_2)\\ \vdots \\  (f, \varphi'_n) \end{bmatrix}. 
\]

\section{Numerical experiments}\label{sec:NR} 

We will compare numerically the standard linear finite element with the $P^1-P^2$-SPLS formulation, and the Streamline Diffusion with Petrov-Galerkin in a variety of norms. In order to compact the tables, we will use the notation $E_{i,method}$ where $i = 0$ is the $L^2$ error $||u-u_h||$, and $i=1$ is the $H^1$ error $|u-u_h|$. For the methods, we have $L$ for standard linear, $S$ for SPLS, $SD$ for Streamline Diffusion, and $P$ for Petrov-Galerkin. 

\subsection{Standard linear  versus SPLS discretization}\label{ssec:NRlin} 
We note here that even in the case when the solution is independent of  $\varepsilon$, the standard finite element solution can exhibit non-physical oscillations, see e.g.. Figure \ref{fig:LinearDino} for the exact solution $u(x)= -x^3 + 1.5 x^2 - 0.5$ and the behavior depends on the parity of $n$-the number of subintervals on $[0, 1]$.

\begin{figure}[h!]
\begin{center}
\includegraphics[width=0.5\textwidth]{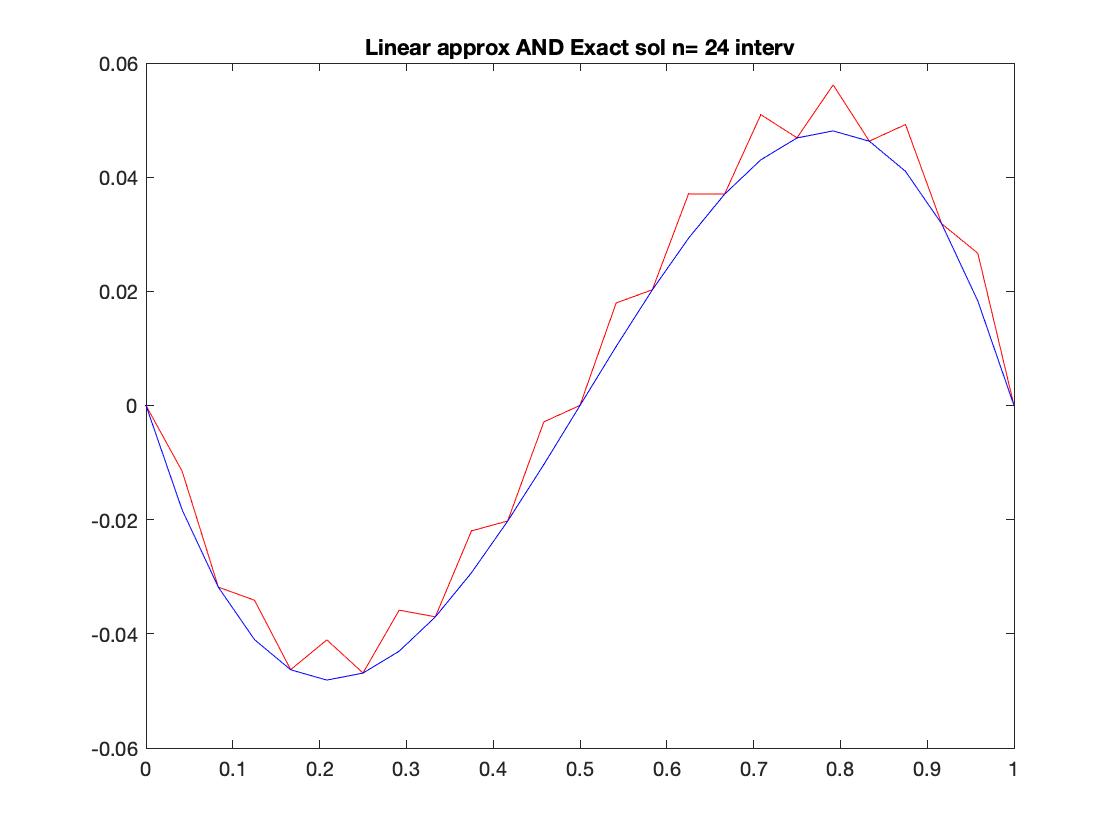}\includegraphics[width=0.5\textwidth]{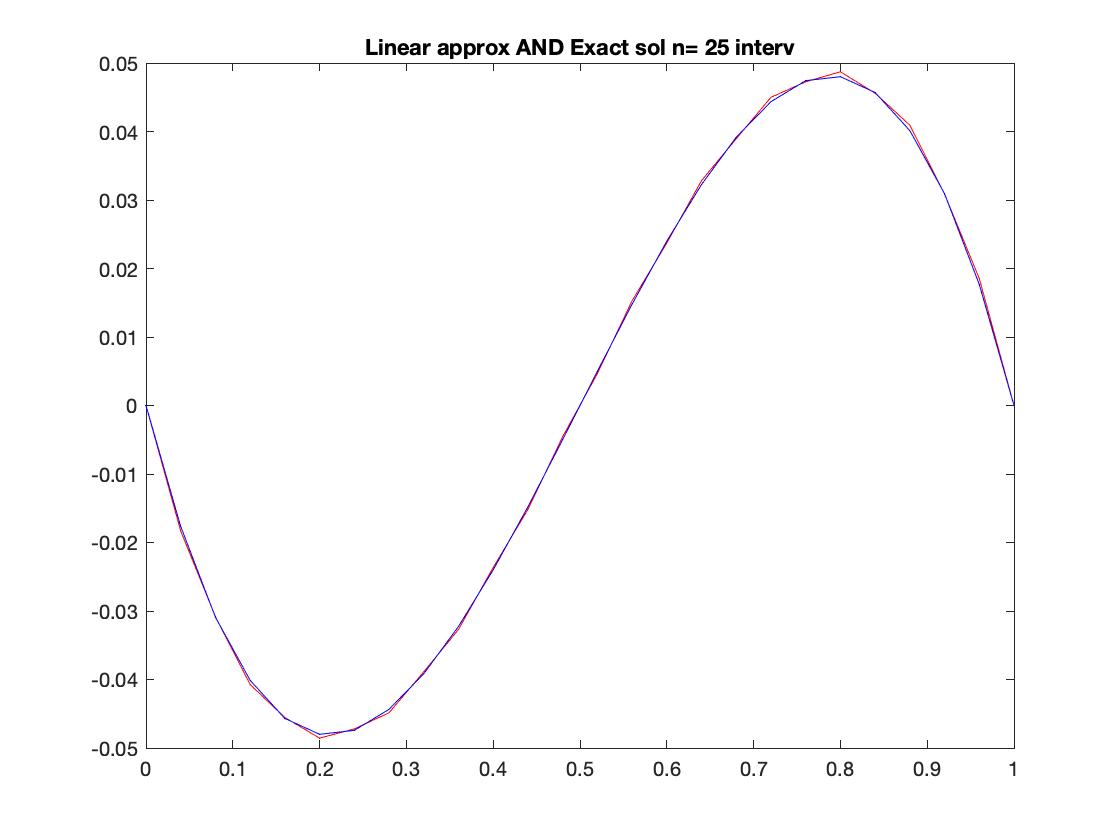}
\end{center}
\caption{ $\varepsilon = 10^{-4}$. Left: $n=24$, Right: $n=25$}
\label{fig:LinearDino}
\end{figure}

For the first test, we take $f = 1-2x$ which satisfies the condition $\overline{f}=0$. We will compare the standard linear finite element method and the SPLS formulation in this case for two values of $\varepsilon$ that are at least 2 orders of magnitude greater than $h$ at the finest level. Table \ref{table:Test1} contains the errors of the two methods over six refinements where $h_i = 2^{-i-5}$. We can see that for this problem, both discretizatin  perform well. The explanation for this nice behavior is that, in the case $\overline{f}=0$, the interpolant has good approximation properties on the  uniform mesh, see the Appendix. We also note that  at all levels for both values of $\varepsilon$ and both errors, SPLS produces smaller error. 

\begin{table}[h!] 
\begin{center} 
\begin{tabular}{|*{5}{c|}} 
\hline 
\multicolumn{1}{|c|}{\multirow{2}{*}{\parbox{1.2cm}{\centering Level/$\varepsilon$}}}& \multicolumn{4}{c|}{$10^{-6}$}\\
\cline{2-5} 
&$E_{1,L}$ &$E_{1,S}$&$E_{0,L}$& $E_{0,S}$\\ \hline 
1&0.289&0.144&0.046&0.011\\ 
\hline 
2&0.144&0.072&0.011&0.003\\ 
\hline 
3&0.072&0.036&0.003&0.001\\ 
\hline 
4&0.036&0.018&0.001&1.8e-4\\ 
\hline 
5&0.018&0.009&1.7e-4&4.4e-5\\ 
\hline 
6&0.009&0.005&4.4e-5&1.0e-5\\ 
\hline 
Order & 1& 1& 2& 2\\
\hline
\multicolumn{1}{|c|}{\multirow{2}{*}{\parbox{1.2cm}{\centering Level/$\varepsilon$}}}& \multicolumn{4}{c|}{$10^{-10}$}\\
\cline{2-5} 
&$E_{1,L}$ &$E_{1,S}$&$E_{0,L}$& $E_{0,S}$\\ \hline 
1&0.289&0.144&0.046&0.011\\ 
\hline 
2&0.144&0.072&0.011&0.003\\ 
\hline 
3&0.072&0.036&0.003&0.001\\ 
\hline 
4&0.036&0.018&0.001&1.8e-4\\ 
\hline 
5&0.018&0.009&1.8e-4&4.5e-5\\ 
\hline 
6&0.009&0.005&4.5e-5&1.1e-5\\ 
\hline 
Order &1 &1 & 2& 2\\
\hline
\end{tabular} 
\caption{L vs. SPLS: $f(x) = 1-2x$} 
\label{table:Test1}
\end{center} 
\end{table}

Table \ref{table:Test1.5} contains errors for standard linear finite elements and SPLS for $f(x) = 2x$ measured in a balanced norm $||\cdot||_B^2 = \varepsilon |\cdot|^2 + ||\cdot||^2$. As this choice of right hand side does not satisfy the condition that $\overline{f} = 0$ we can expect the results to be less impressive than those of Table \ref{table:Test1}. In Table \ref{table:Test1.5} we can see for larger values of $\varepsilon$ the magnitudes of the errors are comparable for both methods. As $\varepsilon$ decreases, while the standard linear elements appear to do better as they attain second order convergence, this is somewhat misleading as the errors are significantly larger than those of SPLS. The SPLS method appears to have a stagnation of error, which is due in part to the overall shift of the approximation which can be seen in the right plot of figure \ref{fig:LinearvsSPLS}. It can be also seen in the previously mentioned figure that SPLS does a better job at capturing the behavior of the exact solution aside from the shift.

Table \ref{table:Test1.5P2} contains errors in the $H^1$, $L^2$, and balanced norms for the SPLS approximation for $f(x)=2x$ with accounting for the expected shift, as presented in Section \ref{sec:MMD}. In that test, the $u_h$ that we measure the error with is taken to be $u_h + \overline{f}/2 = u_h+1/2$. The table shows that for small $\varepsilon$, the shifted SPLS approximation is able to display some convergence order, as anticipated in Section \ref{sec:MMD}. This degeneracy of convergence order may be attributed to the small oscillatory behavior that occurs near both boundaries. The orders improve  if the errors are computed on  the interval $[3h , 1-3h]$, but a rigorous analysis of the shift conjecture and its implications remains to be  investigated.


\begin{table}[h!] 
\begin{center} 
\begin{tabular}{|*{5}{c|}} 
\hline 
\multicolumn{1}{|c|}{\multirow{2}{*}{\parbox{1.2cm}{\centering Level/$\varepsilon$}}}& \multicolumn{4}{c|}{$10^{-4}$}\\
\cline{2-5} 
&$||u-u_{h,L}||_{B}$&Order&$||u-u_{h,S}||_{B}$&Order\\ 
\hline 
1&9.75e-01&0.00&4.97e-01&0.00\\ 
\hline 
2&7.26e-01&0.43&4.91e-01&0.02\\ 
\hline 
3&7.04e-01&0.04&4.77e-01&0.04\\ 
\hline 
4&6.76e-01&0.06&4.86e-01&-0.03\\ 
\hline 
5&6.95e-01&-0.04&5.88e-01&-0.27\\ 
\hline 
6&6.29e-01&0.14&5.76e-01&0.03\\ 
\hline 
\multicolumn{1}{|c|}{\multirow{2}{*}{\parbox{1.2cm}{\centering Level/$\varepsilon$}}}& \multicolumn{4}{c|}{$10^{-8}$}\\
\cline{2-5} 
&$||u-u_{h,L}||_{B}$&Order&$||u-u_{h,S}||_{B}$&Order\\ 
\hline 
1&7.05e+03&0.00&5.02e-01&0.00\\ 
\hline 
2&1.76e+03&2.00&5.01e-01&0.00\\ 
\hline 
3&4.40e+02&2.00&5.01e-01&0.00\\ 
\hline 
4&1.10e+02&2.00&5.00e-01&0.00\\ 
\hline 
5&2.75e+01&2.00&5.00e-01&0.00\\ 
\hline 
6&6.91e+00&1.99&5.00e-01&0.00\\ 
\hline 
\end{tabular} 
\caption{L vs. SPLS: $f(x) = 2x$} 
\label{table:Test1.5}
\end{center} 
\end{table}

\begin{table}[h!] 
	\begin{center} 
		\begin{tabular}{|*{7}{c|}} 
			\hline 
			\multicolumn{1}{|c|}{\multirow{2}{*}{\parbox{1.2cm}{\centering Level/$\varepsilon$}}}& \multicolumn{6}{c|}{$10^{-8}$}\\
			\cline{2-7} 
			&$E_{1,S}$&Order&$E_{2,S}$&Order&$||u-u_h||_B$& Order\\ 
			\hline 
			1&9.35e+00&0.00&6.97e-02&0.00&6.97e-02&0.00\\ 
			\hline 
			2&1.32e+01&-0.50&4.93e-02&0.50&4.93e-02&0.50\\ 
			\hline 
			3&1.87e+01&-0.50&3.49e-02&0.50&3.49e-02&0.50\\ 
			\hline 
			4&2.64e+01&-0.50&2.46e-02&0.50&2.48e-02&0.49\\ 
			\hline 
			5&3.74e+01&-0.50&1.74e-02&0.50&1.78e-02&0.48\\ 
			\hline 
			6&5.30e+01&-0.50&1.23e-02&0.50&1.34e-02&0.41\\ 
			\hline 
		\end{tabular} 
		\caption{SPLS: $f(x) = 2x$ with shift} 
		\label{table:Test1.5P2}
	\end{center} 
\end{table} 

\begin{figure}[h!]
\begin{center}
\includegraphics[width=0.5\textwidth]{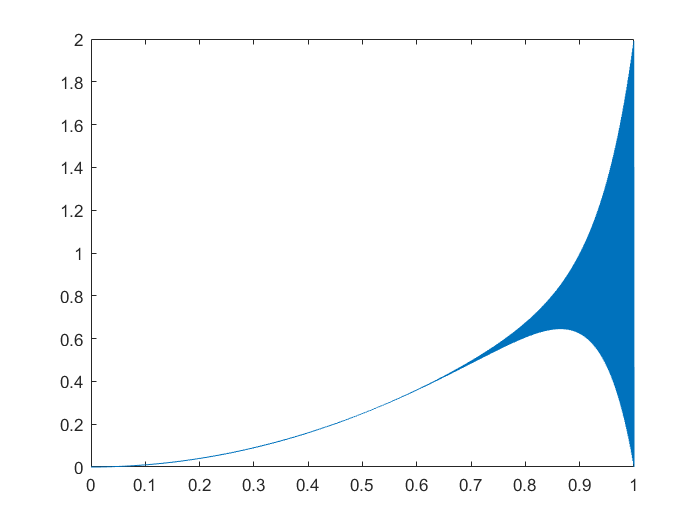}\includegraphics[width=0.5\textwidth]{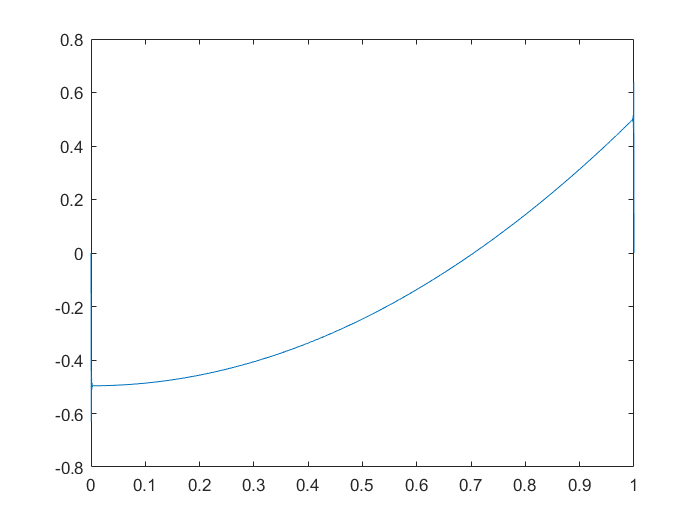}
\end{center}
\caption{ $\varepsilon = 10^{-6}$. Left: Linear, Right: SPLS }
\label{fig:LinearvsSPLS}
\end{figure}

\subsection{Streamline Diffusion  versus PG discretization}


For the second test, we take $f = 2x$ and compare Streamline Diffusion and Petrov-Galerkin. In this case, the exact solution will have a boundary layer at $x = 1$ of width $|\varepsilon\log(\varepsilon)|$. We will also include two tables for this test. Table \ref{table:Test2} compares the errors of the Streamline Diffusion approximation $u_{h,sd}$ with the Petrov-Galerkin approximation $u_{h,pg}$ in the SD norm $||u-u_h||_{sd}$. As we can see in Table \ref{table:Test2}, the expected order for streamline diffusion is observed. Further, the same order is attained by Petrov-Galerkin with errors of smaller magnitude. 

\begin{table}[h!] 
\begin{center} 
\begin{tabular}{|*{5}{c|}} 
\hline 
\multicolumn{1}{|c|}{\multirow{2}{*}{\parbox{1.2cm}{\centering Level/$\varepsilon$}}}& \multicolumn{4}{c|}{$10^{-4}$}\\
\cline{2-5} 
&$||u-u_{h,sd}||_{sd}$ &Order&$||u-u_{h,pg}||_{sd}$&Order\\ 
\hline 
1&1.56e-02&0.00&1.54e-02&0.00\\ 
\hline 
2&2.57e-03&2.60&2.51e-03&2.62\\ 
\hline 
3&5.45e-04&2.24&4.92e-04&2.35\\ 
\hline 
4&1.05e-04&2.37&4.21e-05&3.55\\ 
\hline 
5&3.86e-05&1.45&1.54e-05&1.46\\ 
\hline 
6&1.45e-05&1.41&5.78e-06&1.41\\  
\hline 
\multicolumn{1}{|c|}{\multirow{2}{*}{\parbox{1.2cm}{\centering Level/$\varepsilon$}}}& \multicolumn{4}{c|}{$10^{-8}$}\\
\cline{2-5} 
&$||u-u_{h,sd}||_{sd}$ &Order&$||u-u_{h,pg}||_{sd}$&Order\\ 
\hline 
1&1.46e-02&0.00&1.45e-02&0.00\\ 
\hline 
2&2.16e-03&2.76&2.09e-03&2.79\\ 
\hline 
3&3.98e-04&2.44&3.16e-04&2.72\\ 
\hline 
4&1.01e-04&1.97&4.04e-05&2.97\\ 
\hline 
5&3.60e-05&1.50&1.43e-05&1.50\\ 
\hline 
6&1.27e-05&1.50&5.06e-06&1.50\\ 
\hline 
\end{tabular} 
\caption{SD vs. PG: $f(x) = 2x$}
\label{table:Test2}
\end{center} 
\end{table} 

In Table \ref{table:Test2} and Table \ref{table:Test3},  the SD and PG approximations are  compared in the SD norm $||u-u_h||_{*,h}$, and the balanced norm $||u-u_h||_B$ for $f(x) = 2x$. These tables show that overall, the PG approximation performs better than the SD method for both choices of norms. More interestingly, for the balance norm with small $\varepsilon$ the PG method exhibits  higher order of convergence. 

\begin{table}[h!] 
\begin{center} 
\begin{tabular}{|*{5}{c|}} 
\hline 
\multicolumn{1}{|c|}{\multirow{2}{*}{\parbox{1.2cm}{\centering Level/$\varepsilon$}}}& \multicolumn{4}{c|}{$10^{-4}$}\\
\cline{2-5} 
&$||u-u_{h,sd}||_{B}$ &Order&$||u-u_{h,pg}||_{B}$&Order\\ 
\hline 
1&1.12e-02&0.00&2.28e-03&0.00\\ 
\hline 
2&5.75e-03&0.96&3.97e-04&2.52\\ 
\hline 
3&2.93e-03&0.97&9.84e-05&2.01\\ 
\hline 
4&1.47e-03&0.99&1.13e-05&3.13\\ 
\hline 
5&7.38e-04&0.99&5.61e-06&1.01\\ 
\hline 
6&3.70e-04&1.00&2.80e-06&1.00\\ 
\hline 
\multicolumn{1}{|c|}{\multirow{2}{*}{\parbox{1.2cm}{\centering Level/$\varepsilon$}}}& \multicolumn{4}{c|}{$10^{-8}$}\\
\cline{2-5} 
&$||u-u_{h,sd}||_{B}$ &Order&$||u-u_{h,pg}||_{B}$&Order\\ 
\hline 
1&1.11e-02&0.00&1.60e-03&0.00\\ 
\hline 
2&5.74e-03&0.95&1.62e-04&3.31\\ 
\hline 
3&2.93e-03&0.97&1.65e-05&3.30\\ 
\hline 
4&1.47e-03&0.99&7.00e-07&4.55\\ 
\hline 
5&7.38e-04&0.99&1.82e-07&1.94\\ 
\hline 
6&3.70e-04&1.00&5.16e-08&1.82\\ 
\hline 
\end{tabular} 
\caption{SD vs. PG:$ f(x) = 2x$}
\label{table:Test3} 
\end{center} 
\end{table}





\section{Conclusion}\label{sec:conclusion} 
We compared four discretization methods  for a model convection-diffusion  problem.    Some concepts and observations we noted in the one dimensional case can be used to efficiently discretize and analyze multi-dimensional cases. One such observation is that if the {\it simplified  problem} ($\varepsilon \to 0$)  does not have a unique solution but a {\it particular discretization} we choose of the {\it  simplified problem} has unique solution exhibiting  non-physical oscillations, then the {\it chosen discretization for the original problem} is likely to produce non-physical oscillations.  To eliminate the non-physical solutions one can split the data $f=(f-\overline{f}) + \overline{f} $ and solve the two corresponding  problems for  the data $f-\overline{f}$ and $\overline{f}$. 

For the the model problem we considered, the best method turns out to be the upwinding PG method.  Even though we can view this PG method  as mixed method  with the test space a subspace of $C^0-P^2$-the test space  for  SPLS, the SPLS method is not performing better.  How the upwinding PG method can be extended and related  with other  SPLS  discretizations   in two or more dimensions, will  be  further investigated.   


 \section{Appendix}\label{sec:appendix} 
 We present stability estimates for the model problem  \eqref{eq:1d-model} that justify why in the case of compatibility case $\int_0^1 f(x)\, dx=0$ the standard $C^0-P^1$ or SPLS discretizations  lead to standard  approximation properties, Table 1.  
\subsection{Stability  of the 1D  Convection-Difussion model problem}

The results presented in this section might be well known in a more general setting. However,  we are able to provide   sharp norm estimates   for  the simplified PDE  \eqref{eq:1d-model}.  
We derive estimates for the derivatives  that are used in the next section for establishing approximation properties for the  piece-wise linear interpolant.   All results of this appendix refer to the solution $u=u(x)$ of the  problem \eqref{eq:1d-model}. We assume next  that $f$ is continuous on $[0, 1]$.  The Green's function for this problem allows for the representation
\begin{equation}\label{eq:GreensFnc}
	u(x) = \int_0^1 G(x,s)f(s)\,ds.
\end{equation}
where $G(x,s)$ can be explicitly determined by using standard integration arguments, and 
  \[ 
 G(x,s)=\frac{1}{e^\frac{1}{\varepsilon}-1}\begin{cases}(e^\frac{1}{\varepsilon}-e^\frac{x}{\varepsilon})(1-e^{-\frac{s}{\varepsilon}}),&0\leq s < x\\(e^\frac{x}{\varepsilon}-1)(e^\frac{1-s}{\varepsilon}-1),&x\leq s\leq 1.\end{cases}
 \] 
  Define $u_1(x)$ to be the solution for $f(x) = 1$, or equivalently 
 $$
 u_1(x) = \int_0^1 G(x,s)\,ds = x- \frac{e^{\frac{x}{\varepsilon}} -1}{e^{\frac{1}{\varepsilon}} -1}.
 $$
 
We let $f_{\min}$ and $f_{\max}$  denote the minimum and maximum (respectively) of $f$ on $[0, 1]$, and note that,  
for any  fixed $x\in (0, 1)$, the function 
\[
s\, \to G(x,s), \ s\in [0, 1],
\]
is increasing on $[0, x]$, and decreasing on $[x, 1]$, thus for any $s, x\in [0, 1]$, we have

\begin{equation}\label{eq:bounds4G}
0\leq G(x,s)\leq G(x,x) =\frac{(e^{\frac{1}{\varepsilon}} -e^{\frac{x}{\varepsilon}})(1-e^{\frac{-x}{\varepsilon}})}{e^{\frac{1}{\varepsilon}}-1} \leq \frac{e^{\frac{1}{2\varepsilon}}-1}{e^{\frac{1}{2\varepsilon}}+1}:=G_\infty<1.
\end{equation}

For this problem, we can prove the following inequalities relating the the point values $u(x), u_1(x)$ and $f$.
\begin{theorem}
If \(f\in L^\infty(0,1)\) and \(u\) is the solution to \eqref{eq:1d-model} then:
\begin{enumerate}
\item [i)] \(|u(x)|\leq \|f\|_\infty u_1(x)\);
\item [ii)] \(f_{\min}\, u_1(x)\leq u(x)\leq f_{\max}\, u_1(x)\);
\item [iii)] $|u(x)|\leq G(x,x)  \|f\|_{L^1(0,1)}$ and consequently \\
 \(\|u\|_\infty\leq G_\infty \|f\|_{L^1(0,1)} \leq G_\infty \|f\|_{L^2(0,1)}\).
\end{enumerate}
\end{theorem}
\begin{proof}
The proofs are base on the the definition of $u_1$ and the inequalities of  the Green's function \eqref{eq:bounds4G}.
\begin{enumerate}
\item [i)]  We have:
    \begin{align*}
    |u(x)|&=\left|\int_0^1G(x,s)f(s)ds\right| 
    \leq \int_0^1G(x,s)|f(s)|ds \\
    &\leq \|f\|_\infty \int_0^1G(x,s)ds 
    =\|f\|_\infty u_1(x).
    \end{align*}
\item  [ii)]  Since $f_{\min} \leq f(s)\leq f_{\max}$ we have 
  \[
  \begin{aligned} 
 & f_{\min}\, G(x,s)  \leq f(s)G(x,s)\leq f_{\max}\, G(x,s), \ \text{which implies} \\ 
& f_{\min}\int_0^1G(x,s)ds  \leq \int_0^1f(s)G(x,s)ds \leq f_{\max}\int_0^1G(x,s)ds , \  \text{consequently} \\
    & f_{\min}\, u_1(x)\leq u(x)\leq f_{\max}\, u_1(x).
    \end{aligned}
 \]   
 
\item  [iii)]  First we observe that:
\[
  |u(x)| \leq \int_0^1G(x,s)\, |f(s)|ds  \leq \int_0^1G(x,x)\, |f(s)|ds = G(x,x)\,  \int_0^1 |f(s)|ds.   
 \]  
 Consequently, 
 \[ 
 \|u\|_\infty \leq G_\infty \|f\|_{L^1(0,1)}
 \]
The last part follows from  $\|f\|_{L^1(0,1)} \leq \|f\|_{L^2(0,1)}$.

\end{enumerate}
\end{proof}
\begin{theorem}\label{th:upBound}
If \(u\) is the solution to \eqref{eq:1d-model} and \(f(x)\in C^0([0,1])\) satisfies \(\int_0^1f(s)ds=0\) (i.e. has average 0), then \[|u'(x)|\leq\|f\|_\infty,\quad\forall x\in[0,1].\]
\end{theorem}
\begin{proof}
Using  the explicit form of $G(x,s)$, we have
\begin{align*}
|u'(x)|&=\frac{e^\frac{x}{\varepsilon}}{e^\frac{1}{\varepsilon}-1}\left|\int_0^x\frac{1}{\varepsilon}e^{-\frac{s}{\varepsilon}}f(s)ds+\int_x^1\frac{1}{\varepsilon}e^{\frac{1-s}{\varepsilon}}f(s)ds\right| \\
&\leq \frac{e^\frac{x}{\varepsilon}}{e^\frac{1}{\varepsilon}-1}\left(\left|\int_0^x\frac{1}{\varepsilon}e^{-\frac{s}{\varepsilon}}f(s)ds\right|+\left|\int_x^1\frac{1}{\varepsilon}e^{\frac{1-s}{\varepsilon}}f(s)ds\right|\right).
\end{align*}
Estimating the two integrals
\begin{align*}
\left|\int_0^x\frac{1}{\varepsilon}e^{-\frac{s}{\varepsilon}}f(s)ds\right|&\leq\|f\|_\infty \int_0^x\frac{1}{\varepsilon}e^{-\frac{s}{\varepsilon}}ds \\
&=\|f\|_\infty(1-e^{-\frac{x}{\varepsilon}}) \\
\left|\int_x^1\frac{1}{\varepsilon}e^{\frac{1-s}{\varepsilon}}f(s)ds\right|&\leq\|f\|_\infty\int_x^1\frac{1}{\varepsilon}e^{\frac{1-s}{\varepsilon}}ds \\
&=\|f\|_\infty(e^\frac{1-x}{\varepsilon}-1),
\intertext{ leads to:}
|u'(x)|&\leq\|f\|_\infty\frac{e^\frac{x}{\varepsilon}}{e^\frac{1}{\varepsilon}-1}(1-e^{-\frac{x}{\varepsilon}}+e^\frac{1-x}{\varepsilon}-1) 
=\|f\|_\infty.
\end{align*}
\end{proof}
\begin{corollary}
Under the same assumptions of Theorem 2, we have that 
\begin{equation} \label{eq:upp}
|u''(x)|\leq\frac{2}{\varepsilon}\|f\|_\infty.
\end{equation} 

\end{corollary}
\begin{proof}
Since \(u\) solves \[-\varepsilon u''(x)+u'(x)=f(x),\] we have that:
\begin{align*}
\varepsilon|u''(x)|&\leq |u'(x)|+|f(x)| \\
&\leq \|f\|_\infty+\|f\|_\infty,
\end{align*}
implying the desired result.
\end{proof}

\subsection{Linear interpolant approximation properties} 
For the special  case $\int_0^1 f(x)\, dx=0$ and $f \in C^0([0, 1]) $ we can use the  estimate of theorem \ref{th:upBound}, 
\[
|u'(x)|\leq\|f\|_\infty,\quad\forall x\in[0,1],
\]
to derive an approximation property for the linear interpolant  $\varepsilon$ (assuming that $f$ is independent of $\varepsilon$). 

First we will need an error estimate for the interpolant that does not require the second derivative of the function. 
We will assume  $u \in H^1([0, 1]) $ and $u' \in L^\infty([0,1])$   we consider the linear interpolantc  $0=x_0<x_1<\cdots < x_n=1$ with   $h:=x_j - x_{j-1}, j=1, 2, \cdots, n$. We note first that 
\begin{equation}\label{eq:L2split}
\|u-u_I\|^2_{L^2(0,1)} = \sum_{i=1}^n \int_{x_{i-1}}^{x_i} (u(x) - u_I(x))^2\, dx,  
\end{equation}
and un each interval $[x_{i-1}, x_{i}]$, we have 
\[
u(x) -u_I(x) =\int_{x_{i-1}}^x (u(s) -u_I(s))'\, ds.
\]
Thus, 
\[
\begin{aligned}
(u(x) - u_I(x))^2  & \leq  \int_{x_{i-1}}^x 1^2\, ds  \int_{x_{i-1}}^x (u'(s) - u'_I(s))^2\, ds \\ 
& \leq (x-x_{i-1})  \int_{x_{i-1}}^{x_i} (u'(s) - u'_I(s))^2\, ds.
\end{aligned}
\]
Since  $ \int_{x_{i-1}}^{x_i}  u'(s)   =u(x_i)-u(x_{i-1}) = h u'_I(s)$ we have that 
\[
\begin{aligned}
 \int_{x_{i-1}}^{x_i}  (u'(s) - u'_I(s))^2\, ds  & =  \int_{x_{i-1}}^{x_i}  (u'(s))^2\, ds  - h   (u'_I(s))^2\, ds \\ & \leq  \int_{x_{i-1}}^{x_i}  (u'(s))^2\, ds  \leq h \|u'\|^2_{\infty}. 
 \end{aligned}
\]

Combining the last two estimates,  we obtain the following result:
\begin{equation}\label{eq:L2-estim}
\int_{x_{i-1}}^{x_i} (u(x) - u_I(x))^2\, dx \leq \frac{h^3}{2}  \|u'\|^2_{\infty}. 
\end{equation}
Now, from \eqref{eq:L2split} and  \eqref{eq:L2-estim} we get the following result:
\begin{prop}\label{prop:L2-inteerp} If  $u \in H^1([0, 1]) $ with  $u' \in L^\infty([0,1])$   and  $u_I$ is the linear interpolant on a uniform mesh on $[0, 1]$, then 
\begin{equation}\label{eq:L2err-Int}
\|u-u_I\|^2_{L^2(0,1)} \leq  \frac{h^2}{2}  \|u'\|^2_{\infty}.
\end{equation}
\end{prop}

Assuming now that $f \in C^0([0, 1])$, and  $u$ is the solution of \eqref{eq:1d-model}.  We clearly have  that the regularity assumptions  of Proposition \ref{prop:L2-inteerp}  are satisfied for the solution $u$. Thus, we obtain 

\begin{equation}\label{eq:L2err-Int2}
\|u-u_I\|_{L^2(0,1)} \leq  \frac{h}{\sqrt{2}}  \|f\|_{\infty}, 
\end{equation}
an estimate independent of $\varepsilon$. 

We note here that, as well known from the finite element approximation theory,  this inequality is not  optimal. We can get a standard estimate  $O(h^2)$  for $\|u-u_I\|_{L^2(0,1)}$, at the price of having an estimate constant that depends on  $\varepsilon$. 

First, we note that the following  Poincare Inequality
\begin{equation}\label{eq:P}
 \|w\| \leq \frac {(b-a)}{\pi} \, |w|, \text{ for all }  w \in L^2_0(a,b) \cap H^1(a,b). 
\end{equation}
can be proved using the spectral theorem for compact operators on Hilbert spaces for  the inverse of the (1d) Laplace operator with homogeneous Neumann boundary conditions. 

Next, if $u \in H^2(0,1)$ then using  $\int_{x_{i-1}}^{x_i}  u'(s)   =u(x_i)-u(x_{i-1}) = h u'_I(s)$ and the Poincare inequality \eqref{eq:P},  
\[
\begin{aligned}
 \int_{x_{i-1}}^{x_i}  (u'(s) - u'_I(s))^2\, ds  & =  \int_{x_{i-1}}^{x_i}  \left (u'(s) -  \frac{1}{h} \int_{x_{i-1}}^{x_i}  u'(s) \right)^2\, ds\\
 & \leq \frac{h^2}{\pi^2} \|u''\|_{L^2(x_{i-1}, x_i)}^2.
  \end{aligned} 
\]
Thus,
\begin{equation}\label{eq:L2-estim2}
\int_{x_{i-1}}^{x_i} (u(x) - u_I(x))^2\, dx \leq \frac{h^4}{\pi^2}  \|u''\|_{L^2(x_{i-1}, x_i)}^2,
\end{equation}
which combined with \eqref{eq:L2split} gives 

\begin{equation}\label{eq:L2err-Int4}
\|u-u_I\| \leq  \frac{h^2}{\pi}  \|u''\|_{L^2(0, 1)} 
\end{equation}
Combining with the estimate with \eqref{eq:upp} we obtain 
\begin{equation}\label{eq:L2err-Int4}
\|u-u_I\|_{L^2(0,1)} \leq  \frac{2}{\varepsilon \pi} \, h^2 \|f\|_{\infty}.
\end{equation}

\bibliographystyle{plain} 

\def\cprime{$'$} \def\ocirc#1{\ifmmode\setbox0=\hbox{$#1$}\dimen0=\ht0
  \advance\dimen0 by1pt\rlap{\hbox to\wd0{\hss\raise\dimen0
  \hbox{\hskip.2em$\scriptscriptstyle\circ$}\hss}}#1\else {\accent"17 #1}\fi}
  \def\cprime{$'$} \def\ocirc#1{\ifmmode\setbox0=\hbox{$#1$}\dimen0=\ht0
  \advance\dimen0 by1pt\rlap{\hbox to\wd0{\hss\raise\dimen0
  \hbox{\hskip.2em$\scriptscriptstyle\circ$}\hss}}#1\else {\accent"17 #1}\fi}

 \end{document}